\documentclass[12pt,letterpaper]{article}
\usepackage{graphicx}
\usepackage{amssymb}
\usepackage{amsmath}
\usepackage{bm}
\usepackage{caption}
\usepackage{accents}
\usepackage{standalone}
\usepackage{float}
\usepackage{amsthm}
\usepackage{lineno}
\usepackage{booktabs}
\usepackage{fullpage}
\usepackage{algorithm}
\usepackage{setspace}
\usepackage{tikz,graphics,color,fullpage,float,epsf,caption,subcaption}

\usepackage{tikz-cd}
\usetikzlibrary{positioning}
\usetikzlibrary{calc}
\usetikzlibrary{arrows}
\usetikzlibrary{decorations.markings}
\usetikzlibrary{arrows.meta}
\usetikzlibrary{bayesnet}
\usetikzlibrary{patterns}
\usetikzlibrary{patterns.meta}
\usetikzlibrary{decorations.pathreplacing}  

\usepackage[colorlinks=true, allcolors=blue]{hyperref}
\usepackage{algpseudocode}
\usepackage[utf8]{inputenc}

\newtheorem{theorem}{Theorem}[section]
\newtheorem{proposition}[theorem]{Proposition}

\newtheorem{lemma}[theorem]{Lemma}

\newtheorem{assumption}{Assumption}
\newtheorem{condition}{Condition}
\usepackage{graphicx}

\usepackage{natbib}

\usepackage{pgffor}

\foreach \x in {A,B,...,Z,a,b,...,z} 
{
  \expandafter\xdef\csname cal\x\endcsname{\noexpand 
	\ensuremath{\noexpand\mathcal{\x}}}
  \expandafter\xdef\csname scr\x\endcsname{\noexpand 
	\ensuremath{\noexpand\mathscr{\x}}}
  \expandafter\xdef\csname bb\x\endcsname{\noexpand 
	\ensuremath{\noexpand\mathbb{\x}}}
  \expandafter\xdef\csname rm\x\endcsname{\noexpand 
	\ensuremath{\noexpand\mathrm{\x}}}
  \expandafter\xdef\csname bf\x\endcsname{\noexpand 
	\ensuremath{\noexpand\mbf{\x}}}
}

\newcommand{\tq}[1]{{\textquotedblleft #1\textquotedblright}}

\newcommand{\s}[1][1]{\hspace{#1pt}}

\newcommand{\setst}{\mathbb} % set style
\newcommand{\rcll}{c{\`a}dl{\`a}g}

\newcommand{\rn}{\setst{R}}

\newcommand{\cst}{ \s[0.5] : \s[0.5] }

\begin{document}
\onehalfspacing

%\begin{frontmatter}

\title{Unbiased Simulation Estimators for Multivariate Jump-Diffusions}

\author{Guanting Chen$^\dagger$ \and Alex Shkolnik$^\Diamond$ \and  Kay Giesecke$^\ddagger$}

\date{\small 
$^\dagger$ Institute for Computational and Mathematical Engineering, Stanford University\\
$^\Diamond$ Department of Statistics and Applied Probability, University of California, Santa Barbara\\
$^\ddagger$Department of Management Science and Engineering, Stanford University\\
}

%\address{}
\maketitle

\begin{abstract}
We develop and analyze a class of unbiased Monte Carlo estimators for multivariate jump-diffusion processes with state-dependent drift, volatility, jump intensity and jump size. A change of measure argument is used to extend  existing unbiased estimators for the inter-arrival diffusion to include state-dependent jumps. Under standard regularity conditions on the coefficient and target functions, we prove the unbiasedness and finite variance properties of the resulting jump-diffusion estimators. Numerical experiments illustrate the efficiency of our estimators.

\end{abstract}

\section{Introduction}

The numerical solution of stochastic differential equations (SDEs)  has been a highly active area of research in the applied probability
and Monte Carlo simulation communities.  Historically, the main
emphasis has been placed on the classical case of diffusion
processes with concurrent developments in the literature on
simulation and numerical methods of parabolic equations arising
from Feynman-Kac type formulas.\footnote{The seminal reference
for the simulation of diffusions is \cite{kloeden1999}, with
developments since focused on numerical stability, unbiased
simulation, discontinuous problems, specific models and recent
approaches based on deep learning.  A useful reference for PDE
approaches is \cite{tavella2000} but these typically suffer from
challenges involving the curse of dimensionality.} However, many
applications involve models that are also event-driven in the sense that some of the stochastic uncertainty is represented by jumps. Indeed, a large
literature on SDEs with
jumps exists to address the theory and applications. Such SDEs
commonly arise in finance, economics, insurance, epidemiology,
chemistry and other areas. However, the literature on simulation
methods for jump-diffusions has received significantly
less attention than the more classical diffusion counterpart.

This paper develops simulation estimators for multi-variate jump-diffusion processes with state-dependent coefficients for the
drift, diffusion and jump characteristics. That is, we 
consider $\bbR^d$-valued Markov processes solving
the SDE
\begin{align} \label{itosde}
d X_t = \mu(X_t) dt + \sigma(X_t) d W_t + 
\int_{\bbR^d_+}c(X_{t-}, z)\s M(dt, dz)
\end{align}
where $W$ is a standard Brownian motion and $M$ is a Poisson
random measure. The drift and diffusion coefficient functions $\mu$
and $\sigma$ are associated with the classical setting described
above. The third term imparts jumps into the dynamics as
governed by the measure $M$ and $c$, the jump coefficient
function. The jumps need not be distributed according to a
Poisson process as the jump coefficient $c$ is state dependent.

%While certain regularity assumptions on the coefficients
%$\mu,\sigma,c$ are requited by the theoretical analysis, our main contribution is a generic (i.e.  model free) algorithm for
%the simulation of a jump-diffusion process.  Typically, generic
%simulation algorithms for SDEs belong to the domain of
%discretization methods, which require only the evaluation of the
%coefficients at discrete points in time. 

A standard approach to simulating (\ref{itosde}) is discretization. While broadly applicable and easy to implement, discretizaton
methods generate biased simulation estimators and their error
analysis always entails additional assumptions on the coefficient functions.\footnote{The
classic references for discretization methods for diffusions
with jumps is \cite{platen2010}. See also \cite{shkolnik2021}
for more recent results.} Simulation bias is undesirable not only due to
the fact that the solution is approximate but also due to issues
of numerical stability that algorithms can exhibit for certain
models or parameter ranges.
To address these issues, several approaches to unbiased
simulation have been developed in recent decades.	The most
ambitious of these is exact sampling. Exact acceptance-rejection schemes have been developed for one-dimensional diffusions by \cite{beskos2005} \footnote{See also \cite{beskos2006} and \cite{chen2013} for extensions.}. Extensions of the former to one-dimensional SDEs with jumps have been developed in
\cite{casella2011}, \cite{giesecke2013},  \cite{goncalves2014}, and \cite{pollock2016exact} under various sets of assumptions. Recently, new techniques developed by \cite{blanchet2020} led to the first exact scheme for multivariate diffusions \footnote{See also \cite{blanchet2017} for Tolerance-Enforced Simulation for multivariate diffusion process.}. The multivariate setting however presents unique challenges; for
example, the run time of the algorithm in \cite{blanchet2020} is
infinite in expectation. An
alternative to exact sampling is unbiased
estimation, which does not involve exact samples, but nevertheless yields simulation estimators for functions of the process that are free of bias.  This is the approach we pursue in this paper.

For an $\bbR^d$-valued stochastic process $X$, an objective
function $f \cst \bbR^d \to \bbR$, and some time horizon $T >
0$, an unbiased estimator of $f(X_T)$ is a random variable
$\Xi$ such that $\mathbb{E}[\Xi] = \mathbb{E}[f(X_T)]$. We remark that
$\Xi$ is not required to be an exact sample of $f(X_T)$, and
for example, $\Xi$ may be negative all the while $f$ is
positive valued. %Although this type of estimator has
%shortcomings in some applications, it nevertheless guarantees
%that no bias is introduced into the simulation.  
\cite{glynn2015}, \cite{ballyhiga2015}, \cite{agarwal2017}, \cite{doumbia2017},
\cite{anderssonhiga2017},  \cite{labordere2017} and \cite{chen2020} have developed and analyzed unbiased estimators for diffusion processes. These estimators are based on ideas that draw from the literature
on multi-level Monte Carlo and formulations that
date back to the study of partial differential equations
in \cite{levi1907}. An earlier effort by \cite{wagner1989} is based on 
solutions to integral equations via the von Neumann--Ulam 
scheme.
The extension of these diffusion approaches to include state-dependent jumps is not obvious. The
difficulties trace to the particular form of the infinitesimal
generator of a jump-diffusion, which has properties that
distinguish it from the partial differential operators
arising in the diffusion case.
For this reason new approaches are required for the design of
unbiased simulation estimators for multivariate jump-diffusions  (\ref{itosde}).

We construct unbiased simulation estimators for (\ref{itosde}) from existing unbiased diffusion estimators. Our approach entails a change of measure that 
facilitates the 
exact sampling of the jump times of the process. Specifically, under the sampling
measure, the jump inter-arrival times are exponentially
distributed. The sampling measure further possesses a
convenient conditional independence property that
preserves the law of the diffusion process between the jump
times. As a consequence, any sampling approach
may be used to generate the diffusion skeletons.
This allows for a \tq{black-box} implementation in which 
the next jump time of the process is sampled first, and 
then virtually any existing diffusion scheme may be used to sample 
the diffusion on the generated time interval. For our
purposes, the black-box is any unbiased simulation
estimator for a multivariate diffusion. We provide sufficient
conditions for a diffusion estimator to be extensible. Under standard regularity conditions on the coefficient and objective functions for (\ref{itosde}), we prove the unbiasedness and finite variance properties of our jump-diffusion estimators. 

Numerical experiments indicate that our estimators are significantly more efficient than existing exact sampling and discretization estimators. In the one-dimensional special case treated by \cite{chen2019}, our scheme runs faster by a factor of $100+$ than the exact sampling scheme of \cite{giesecke2013}. In this paper, we run our scheme against the discretization scheme of 
\cite{shkolnik2021} for two multivariate models, one
meeting our technical hypotheses and the other one violating them. The results indicate the superior performance of our estimator in both cases.

The rest of the paper is organized as follows: In Section 2 we
introduce the problem. In Section 3 we develop  our change-of-measure approach of extending unbiased diffusion estimators to jump-diffusions.  In Section 4 we
illustrate our approach for the ``parametrix''
diffusion estimator; we modify the original regularity
conditions to enable extensibility of the estimator. In
Section 5 we perform numerical experiments. Appendices contain the proofs.

\section{Formulation} \label{S:2}
The goal of this paper is to develop an unbiased estimator $\Xi$ such that $\mathbb{E}[\Xi] = \mathbb{E}[f(X_T)]$, where $X \in \mathbb{R}^d$ is a jump-diffusion process solving \eqref{itosde} and $T > 0$ is the time horizon. We fix a complete probability space $(\Omega,\mathbb{P},\mathcal{F})$ equipped
with a filtration $\mathbb{F} = (\mathcal{F}_t)_{t \ge 0}$ satisfying the usual conditions \citep{protter2005}. For integers
$m,d \ge 1$, a standard $m$-dimensional
Brownian motion $W$ and functions
$\mu \cst \bbR^d \to \bbR^d$, $\sigma \cst \bbR^d \to 
\bbR^{d \times m}$ and $c \cst \bbR^d \times \bbR^d_+
\to \bbR^d$ we write \eqref{itosde} in the form
\begin{equation}\label{s1_JD_sde}
\begin{aligned}
X_t = X_0 + \int_0^t\mu(X_s) ds + 
\int_0^t\sigma(X_s) dW_s 
+ \sum_{n=1}^{N_t} h(X_{T_n-},R_n)
\end{aligned}
\end{equation}
for $t \in [0, T]$, some function $h \cst \bbR^d \times \bbR^d
\to \bbR^d$, a counting process $N$ with jump times
$(T_n)_{n \ge 1}$ and a sequence $\{R_n\}_{n \ge 1}$
of random variables where every $R_n \in \bbR^d$ is
independent of $X_{T_{n^-}}$ and is distributed according to a law
$\nu$. The intensity 
(or conditional arrival rate) of the process $N$ 
at time $t$ is given
by a random measure assigned to the set $\{ z \in \bbR^d_+
\cst c(X_{t-}, z) \neq 0\}$ for $c$ in \eqref{itosde}.
The existence and uniqueness of the process $X$ 
is guaranteed \citep[Theorem 3.8]{cinlar2011} under the 
assumptions on the coefficients that we impose below.

The intensity of $N$ may be defined as $\lambda(X)$ for
some function $\lambda \cst \bbR^d \to [0,\infty)$ and we
will assume $\lambda$ to be bounded. The jump-diffusion
$X$ may be constructed iteratively, by evolving a diffusion
$Y$ up to its killing time that arrives with rate 
$\lambda(Y)$. At each such time the process incurs
a state dependent jump governed by the function $h$ and
the diffusion regenerates. Next we make this construction 
precise.

We define a process $Y$ on $[0,T]$ as a solution to the SDE
\begin{align} \label{ysde}
  dY_t = \mu (Y_t) dt + \sigma (Y_t) dW_t, \,\,\,\,\, Y_0 = x
\end{align}
for the same functions $\mu,\sigma$ and $W$ defined in $\eqref{s1_JD_sde}$. A (weak) solution of $\eqref{s1_JD_sde}$ may be constructed from
 i.i.d. copies $\{W^n\}_{n \in \mathbb{N}}$ of the $W$
and a sequence of standard exponential random variables 
$\{\mathcal{E}_n\}_{n \in \mathbb{N}}$. To this end, 
for the intensity function $\lambda$, define $A$ as
\begin{align} \label{ycomp}
   A_t = \int_0^t \lambda (Y_s) ds,
\end{align}
and take $(Y^n,A^n)$ to be defined via
$\eqref{ysde}$ and $\eqref{ycomp}$ but with respect to $W^n$.
This pair corresponds to the interval $[T_n,T_{n+1})$
with the right endpoint given by the relation 
$T_{n+1} = T_n + (A^n)_{\mathcal{E}_n}^{-1}$. 
Now, starting at $T_0 = 0$ and $X_{0^-} =0$, we proceed as,
\begin{equation} \label{construct}
\begin{aligned}
 X_{T_{n}} &= X_{T_{n}-} + h(X_{T_n^-}, R_n)  \\
 X_t &= Y^{n}_{t - T_{n}};
  \hspace{0.16in} t \in (T_n, T_{n+1}), \hspace{0.16in} Y^n_0 = X_{T_n}.
\end{aligned}
\end{equation}
A solution $X$ that follows the above 
recipe is {\rcll} and enjoys the strong Markov property at each stopping time $T_n$.

We will use this solution to construct the unbiased estimator for $\mathbb{E}[f(X_T)]$. Before we introduce the main results, a few assumptions are stated below, where we denote $C^1_b(\mathbb{R}^d)$ to be the set of functions in $\mathbb{R}^d$ that is bounded and continuously differentiable.

\begin{assumption} \label{dasm}
We have the drift function $\mu \in C^1_b (\mathbb{R}^d)$.
The diffusion matrix $\sigma\sigma^T \in C^1_b (\mathbb{R}^d)$ and is uniformly elliptic\footnote{ For all $y \in \mathbb{R}^d$, $\exists$ $a_{\max} > a_{\min} > 0$ s.t. $a_{\min}|y|^2 \leq y^T\sigma(X_t(\omega))\sigma(X_t(\omega))^Ty \leq a_{\max}|y|^2$.}.
\end{assumption}

\begin{assumption} \label{jasm}
The intensity function $\lambda \in C_b^1 (\mathbb{R}^d)$ and there exist $\lambda_2 > \lambda_1 > 0$ such that $\lambda_1 \leq \lambda(x) \leq \lambda_2$ for all $x \in \mathbb{R}^d$. Moreover, the function
$h(x, R_n)$ is uniformly bounded such that $||h||_{\infty} \leq v$. Lastly, $h(x, R_n)$ has an uniformly bounded finite second moment and we can sample $R_n$ directly from its distribution $\nu$.
\end{assumption}

Assumption \ref{dasm} and the bounded intensity of Assumption \ref{jasm} guarantee the existence and uniqueness of
the solution of $\eqref{ysde}$ and $\eqref{s1_JD_sde}$.  Assumption \ref{jasm}
ensures that $X$ solving $\eqref{s1_JD_sde}$ is
non-explosive.  Bias free samples and a finite second moment of $h(x, R_n)$ are required for constructing our jump-diffusion estimator.

The intuition for deriving our estimator is that, between every jump time, for $t \in (T_{i-1}, T_i)$, we approximate $X_t$ (which is also $Y_{t-T_{i-1}}$) by an Euler process $Y^\pi$ and a correction functional such that we have unbiased estimator for functions of diffusions. Knowing that within each section $(T_{i-1}, T_i)$ we can get unbiased estimation, the challenge becomes how to make jump time analytically tractable. In the next section, we are going to state the main result based on this formulation.

\section{Estimator}
\subsection{Main Results}
In this section we present the main results. For clarity, we first introduce the notion of unbiased diffusion estimator. Generally speaking, if we consider a diffusion process $Y$ defined in $\eqref{ysde}$, for any sequence $\{t_i\}_{i=1}^m$ such that $t_i < t_{i+1}$ and $t_i \in (0, T)$, we define an Euler approximation process $Y^\pi$ such that
\begin{equation*}
    \begin{aligned}
    Y_t^{\pi} = Y_{t_i}^{\pi} + \mu(Y_{t_i}^{\pi})(t-t_i) + \sigma(Y_{t_i}^{\pi})W_{t-t_i}^{i},\,\,\,\,\,t\in(t_i, t_{i+1}],\,\,\,\,\,Y_0^{\pi} = x.
    \end{aligned}
\end{equation*}
A number of unbiased diffusion estimators have been developed in the literature \citep{wagner1989,ballyhiga2015,labordere2017,agarwal2017,doumbia2017}. More specifically, for a diffusion process $Y$ and bounded function $f: \mathbb{R}^n \to \mathbb{R}$, those estimators feature an Euler approximation $Y^\pi$ and a correction functional $\Theta:\left(\mathbb{R}^{n}\times [0, T]\right)\times\mathbb{R}\to \mathbb{R}$ such that
\begin{equation}\label{eqn_diffusion}
\begin{aligned}
\mathbb{E}[f(Y_T)] = \mathbb{E}[f(Y_T^\pi)\Theta(Y^{\pi},T)].
\end{aligned}
\end{equation}
The differences between those estimators mentioned above are the way they sample the Euler process $Y^\pi$ and the way they define $\Theta$. Our jump-diffusion estimator could be derived from any of the diffusion estimators mentioned above. Henceforth, we will refer to estimators in the above form as ``black-box'' estimators.

For the jump-diffusion process, we add the jump intensity process $\lambda(X)$ as another dimension to the diffusion formulation \eqref{ysde}. From the weak solution construction \eqref{construct}, we know that between jump times the jump-diffusion process has the same dynamics as the diffusion process. Therefore, with a little abuse of notation, by ignoring the superindex in $\eqref{construct}$, we define the following auxiliary diffusion process $Z = (Y,\bar{A})$, where $\sigma_A > 0$ is chosen by design and $\bar{W}$ is a $d$-dimensional Brownian motion independent of $W$.
\begin{equation} \label{eqn_sde_coupled}
\begin{aligned}
dY_t &= \mu(Y_t)dt + \sigma(Y_t)dW_t\\
d\bar{A}_t &= \lambda(Y_t)dt + \sigma_Ad\bar{W}_t\\
Y_0 &= x
, \,\,\bar{A}_0 = 0.
\end{aligned}
\end{equation}
% Therefore, with a little abuse of notation, we denote $Z_{x}$ to indicate the fact the component $Y$ starts from $x$ (since $\bar{A}$ always starts from $0$).
From above, we know the process $Y$ starts from $x$ and the process $\bar{A}$ starts from $0$.  Throughout, $\mathbb{E}_{x}$ denotes taking expectation conditioned on $Y$ (or its jump-diffusion counterpart $X$) starting from $x \in \mathbb{R}^d$. 

The reason we introduce the other dimension $\bar{A}$ is to incorporate the change of measure, which helps to sample the jump times. To measure the effect of measure change, we define the processes $L_i(Z)$ by 
\begin{equation} \label{L_process}
\begin{aligned}
L_1(Z)_t &= \exp\left(-\bar{A}_t + t\lambda(Y_0)\right)\frac{\lambda(Y_t)}{\lambda(Y_0)}\\
L_2(Z)_t &= \exp\left(-\bar{A}_t + t\lambda(Y_0)\right).
\end{aligned}
\end{equation}

As will be shown later, with $L_i(Z)$ we can have the measure changed such that the jump time can be sampled more efficiently. However, incorporating $L_i(Z)$ in the objective function adds an exponential term, which requires additional efforts to address the regularity issue (recall that the black-box estimator requires the objective function to be bounded).

We say $f : \mathbb{R}^d \to \mathbb{R}$ has exponential growth if there are constants $c_1, c_2 > 0$ such that $|f(x)| \leq e^{c_1||x||_1+c_2}$ for all $x \in \mathbb{R}^d$. We require the following condition for the black-box diffusion estimator to accommodate the construction of jump-diffusion estimator.

\begin{condition} \label{cdn_blackbox}
Let $Z$ be a $d+1$-dimensional diffusion process defined in \eqref{eqn_sde_coupled}. 
$g : \mathbb{R}^{d+1} \to \mathbb{R}$  is a function of the form $g(z) = \exp(-a)f(y)$, where $z = (y, a)$ and $f: \mathbb{R}^d \to \mathbb{R}$ is a function with exponential growth.
There exists a black-box algorithm which takes an Euler approximation $Z^{\pi}$ of the process $Z$ and a correction functional $\Theta:\left(\mathbb{R}^{n+1}\times [0, T]\right)\times\mathbb{R}\to \mathbb{R}$ such that for $T > 0$, we have
\begin{equation}\label{eqn_cdn_blackbox}
\begin{aligned}
\mathbb{E}[g(Z_T)] = \mathbb{E}[\exp(-\bar{A}_T)f(Y_T)] = \mathbb{E}[\exp(-\bar{A}_T^\pi)f(Y_T^{\pi})\Theta(Z^{\pi},T)].
\end{aligned}
\end{equation}
\end{condition}
Condition \ref{cdn_blackbox} is a more general version of \eqref{eqn_diffusion} with less restrictive regularity condition such that $g(z) = \exp(-a)f(y)$ is not bounded. Another important feature for a valid Monte Carlo estimator is the finite variance property.  We require the following:
% To meet this condition, \cite{chen2020} extend the results of \citet{ballyhiga2015} and \citet{anderssonhiga2017} such that $g$ can have exponential growth, and thus facilitates our construction.

\begin{condition} \label{cdn_var}
For $f : \mathbb{R}^d \to \mathbb{R}$ with exponential growth, and $Z^{\pi}$ be the Euler process of a black-box estimator in Condition \ref{cdn_blackbox}. There exists constant $M_T$ such that:
\begin{equation*}
\begin{aligned}
\mathbb{E}[\left(\exp\left(-\bar{A}_T^{\pi}\right)f(Y^\pi_T)\Theta(Z^{\pi}, T)\right)^2] < M_T\exp\left(2c_1||x||_1\right),
\end{aligned}
\end{equation*}
where $x = (x^{(1)}, \cdots x^{(d)})$ and $(x,0)$ is the starting point of $Z^\pi$.
\end{condition}
% Notice that Condition \ref{cdn_var} also implies
% \begin{equation}\label{ineq_cosh}
% \begin{aligned}
% \mathbb{E}[\left(\exp\left(-\bar{A}_T^{\pi}\right)f(Y^\pi_T)\Theta(Z^{\pi}, T)\right)^2] < 2M_T\prod_{i=1}^d \cosh(2c_1x^{(i)}),
% \end{aligned}
% \end{equation}
% where $\cosh(x) = (e^x + e^{-x})/2$ is the hyperbolic cosine function. \eqref{ineq_cosh} features better analytical trackability and will be used in the later proof.
Condition \ref{cdn_var} characterizes the moment condition for the black-box estimator. With Condition \ref{cdn_blackbox} and \ref{cdn_var}, by denoting
\begin{equation}\label{L_process_theta}
\begin{aligned}
L_i^{\Theta}(Z^{\pi})_T = L_i(Z^{\pi})_T\Theta(Z^{\pi}, T),
\end{aligned}
\end{equation}
we have the following results.
\begin{proposition}\label{prop_two_estimators}
Let $f : \mathbb{R}^d \to \mathbb{R}$ be a function of exponential growth, $Z = (Y, \bar{A})$ be the process defined in \eqref{eqn_sde_coupled} and $Z^{\pi} = (Y^{\pi}, \bar{A}^{\pi})$ be the Euler process of the black-box estimator satisfying Condition \ref{cdn_blackbox} and \ref{cdn_var}. $T$ and $T_1$ are constants such that $T_1 \leq T$. The estimators $L_1^{\Theta}(Z^{\pi})_{T_1}$ and $L_2^{\Theta}(Z^{\pi})_Tf(Y_T^{\pi})$ have the property that
\begin{equation*}
    \begin{aligned}
    \mathbb{E}[L_1^{\Theta}(Z^{\pi})_{T_1}] = \mathbb{E}_{x}[L_1(Z)_{T_1}],\,\,\quad
    \mathbb{E}[L_2^{\Theta}(Z^{\pi})_Tf(Y_T^{\pi})] = \mathbb{E}_{x}[L_2(Z)_Tf(Y_T)].
    \end{aligned}
\end{equation*}
\end{proposition}

The estimators defined in the above proposition yield the theorem for our final jump-diffusion estimator.

\begin{theorem} \label{theorem1_recursive}
Let $X_T$ be the process defined in \eqref{s1_JD_sde}, $Z^\pi$ be the Euler process of the black-box estimator satisfying Condition \ref{cdn_blackbox} and \ref{cdn_var}, and
$\xi_1$ be an independent exponential random variable with rate $\lambda(x)$. Define $U(x, T)$ by the following recursive equation
\begin{equation} \label{recursive}
\begin{aligned}
U(x, T) &= 1_{\{\xi_1\geq T\}}\Xi_2(x, T, Z^{\pi}) + 1_{\{\xi_1< T\}}\Xi_1(x, \xi_1, Z^{\pi})U(Y_{\xi_1}^{\pi} + h(Y_{\xi_1}^\pi, R_1), T-\xi_1),
\end{aligned}
\end{equation}
where
\begin{equation}\label{eqn_xi}
\begin{aligned}
\Xi_1(x, \xi_1, Z^{\pi}) &= \exp\left({-\sigma_A^2\xi_1/2}\right)L_1^{\Theta}(Z^{\pi})_{\xi_1}\\
\Xi_2(x, T, Z^{\pi}) &= \exp\left({-\sigma_A^2T/2}\right)L_2^{\Theta}(Z^{\pi})_{T}f(Y_T^{\pi}).
\end{aligned}
\end{equation}
Then, under Assumption \ref{dasm}-\ref{jasm} and Condition \ref{cdn_blackbox}-\ref{cdn_var}, $U(x, T)$ is an unbiased estimator for the functional of $X_T$ such that $\mathbb{E}[U(x, T)] = \mathbb{E}_{x}[f(X_T)]$. Moreover, $U(x, T)$ has finite variance.
\end{theorem}

Theorem \ref{theorem1_recursive} states that as long as the black-box diffusion estimator meets Condition \ref{cdn_blackbox} and \ref{cdn_var}, we can extend this estimator to accommodate jumps. 

Notice that the term $\exp(-\sigma_A^2T/2)$ appears because in the definition of $L_1$ and $L_2$ \eqref{L_process}, the expectation is $\mathbb{E}[\exp(-\bar{A})] = \exp(\sigma_A^2T/2) \mathbb{E}\left[\exp\left(-\int\lambda(Y_s)ds\right)\right]$. However, as will be shown later, for unbiased estimation, the major quantity of interest is $\mathbb{E}\left[\exp\left(-\int\lambda(Y_s)ds\right)\right]$. Henceforth, we have to multiply $\exp(-\sigma_A^2T/2)$ in order to compensate the independent Gaussian noise $\sigma_A\bar{W}_T$.

The algorithm for our jump-diffusion estimator can be deduced directly from Theorem \ref{theorem1_recursive}. But firstly we need some notations. We define $T^{\xi}_i = \sum_{j=1}^i \xi_{j}$. Notice that $T^{\xi}_i$ could be interpreted as the $i$-th jump time in our estimator, where the jump times are sampled from exponential distributions. We also define $Z^{\pi, i} = (Y^{\pi, i} \bar{A}^{\pi, i})$, which could be interpreted as an approximation process of $Z_t$ for $t \in (0, T^\xi_{i} - T^\xi_{i-1})$. More specifically, for $t\in (0, T^\xi_{i} - T^\xi_{i-1})$,  $Y^{\pi, i}_t$ could be interpreted as the Euler approximation of the weak solution of $X$ in \eqref{construct} (for $t \in (T^\xi_{i-1}, T^\xi_{i})$), starting at the ending point of $Y^{\pi, i-1}$ plus the jump size, i.e. $$Y^{\pi, i}_0 = X_{T^{\xi}_{i-1}}^{\pi,i-1} = Y_{\xi_{i-1}}^{\pi,i-1} + h(Y_{\xi_{i-1}}^{\pi,i-1}, R_{i-1}).$$

For cleaner notation, we denote $V_i = h(Y_{\xi_{i}}^{i}, R_{i})$ and $V_i^{\pi} = h(Y_{\xi_{i}}^{\pi, i}, R_{i})$. Algorithm \ref{alg:unbiased} details the jump-diffusion scheme.

\begin{algorithm}[ht!]
\caption{Black-Box Jump-Diffusion Estimator}
\label{alg:unbiased}
\begin{algorithmic}[1]
 \State Choose a black-box estimator that has the correction functional and Euler approximation pair $(\Theta, Z^{\pi})$ satisfying Condition \ref{cdn_blackbox} and \ref{cdn_var}. 
 \State Initialize $M = 1$, $i = 1$. Sample exponential arrival times $T^\xi_1 = \xi_1$ with intensity $\lambda(x_0)$.
 \While{$T_i^{\xi} < T$}
\State Simulate $Z^{\pi,i} = (Y^{\pi,i}, \bar{A}^{\pi,i})$ with starting point $(x_{i-1}, 0)$. 
\State Simulate $V_i^{\pi} = h(Y_{\xi_{i}}^{\pi, i}, R_{i})$. 
\State Compute $L_1^{\Theta}(Z^{\pi,i})_{\xi_i}$, update $M \gets ML_1^{\Theta}(Z^{\pi,i})_{\xi_i}$, and compute $x_i = Y_{\xi_i}^{\pi,i} + V_i^{\pi}$.
\State Sample $\xi_{i+1}$ from exponential arrival time with intensity $\lambda(x_i)$, and update $i\gets i+1$.
 \EndWhile
\State Sample $Z^{\pi,i}$ with starting point $(x_{i-1},0)$. Compute $L_2^{\Theta}(Z^{\pi, i})_{T-T^{\xi}_{i-1}}$ return
$$\exp(-\sigma_A^2T/2)M L_2^{\Theta}(Z^{\pi,i})_{T - T_{i-1}^{\xi}}f\left(Y^{\pi,i}_{T-T_{i-1}^{\xi}}\right).$$
\end{algorithmic}
\end{algorithm}

\subsection{Estimator Derivation}
The main idea is to use change of measure and an iterative Monte-Carlo approach. One major difficulty for simulating jump-diffusion model is that the jump time is related to the intensity, which could be a function of the jump-diffusion process itself. We circumvent this by changing measure from $\mathbb{P}$ to $\mathbb{Q}$ wherein the intensity will be constant between jump-times. To do so, we define a {\rcll} process $L(X)$ by
\begin{equation*}
\begin{aligned}
L(X)_t = \exp\left(\int_0^t-\left(\lambda(X_s) - \lambda(X_{T_{N_s}})\right)ds\right)\prod_{n=1}^{N_t}\frac{\lambda(X_{T_n^-})}{\lambda(X_{T_{n-1}})}.
\end{aligned}
\end{equation*}
Theorem 3.1 in \citet{bias2021} guarantees the existence of $\mathbb{Q}$ via the Radon-Nikodym derivative $L(X)_T$ of $\mathbb{P}$ with respect to $\mathbb{Q}$. Under $\mathbb{Q}$, the jump time $T_{n+1} - T_{n}$ follows an exponential distribution of parameter $\lambda(X_{T_n})$. From the construction $\eqref{construct}$, we also know that the strong Markov property holds under measure $\mathbb{Q}$.

Then we start to write down $E_x[f(X_T)]$ in a form that leads to the estimator:
\begin{equation} \label{eqn_expansion}
\begin{aligned}
\mathbb{E}_{x}[f(X_T)] &= \mathbb{E}_{x}^{\mathbb{Q}}[L(X)_Tf(X_T)]\\ 
&= \mathbb{E}_{x}^{\mathbb{Q}}[L(X)_Tf(X_T)1_{\{T_1 < T\}}] + \mathbb{E}_{x}^{\mathbb{Q}}[L(X)_Tf(X_T)1_{\{T_1 \geq T\}}]\\
&= \mathbb{E}_{x}^{\mathbb{Q}}[\mathbb{E}_{x,\mathcal{F}_{T_1}}^{\mathbb{Q}}[L(X)_{T_1}\frac{L(X)_{T}}{L(X)_{T_1}}f(X_T)1_{\{T_1 < T\}}]] + \mathbb{E}_{x}^{\mathbb{Q}}[L(X)_Tf(X_T)1_{\{T_1 \geq T\}}]\\
&= \mathbb{E}_{x}^{\mathbb{Q}}[1_{\{T_1 < T\}}L(X)_{T_1}\mathbb{E}_{X_{T_1}}[f(X_T)]] + \mathbb{E}_{x}^{\mathbb{Q}}[L(X)_Tf(X_T)1_{\{T_1 \geq T\}}]\\
& = \mathbb{E}_{x}^{\mathbb{Q}}[1_{\{T_1 < T\}}L(X)_{T_1}\mathbb{E}_{X_{T_1}}[f(X_T)] + 1_{\{T_1 \geq T\}}L(X)_Tf(X_T)].
\end{aligned}
\end{equation}
Let $\{{\xi_i}\}_{i=1}^{\infty}$ be an sequence of exponential random variables with rate $\{\lambda(X_{T_{i-1}})\}_{i=1}^{\infty}$. From the change of measure argument we know that $T_n = \sum_{i=1}^n\xi_i$ a.s. under measure $\mathbb{Q}$. From \eqref{eqn_expansion} we know that if we come up with estimators for $\mathbb{E}^{\mathbb{Q}}\left[1_{\{T_1 < T\}}L(X)_{T_1}\mathbb{E}_{X_{T_1}}[f(X_T)]\right]$ and  $\mathbb{E}^{\mathbb{Q}}\left[1_{\{T_1 \geq T\}}L(X)_Tf(X_T)\right]$, an unbiased estimator for $\mathbb{E}_{x}[f(X_T)]$ could be derived.

We start by deriving an estimator for $\mathbb{E}^{\mathbb{Q}}\left[1_{\{T_1 \geq T\}}L(X)_Tf(X_T)\right]$. Firstly, under measure $\mathbb{Q}$ and event ${\{T_1 \geq T\}}$, the law of the Brownian Motion driving $X$ and $Y$ in $\eqref{construct}$ is unchanged. Hence we have
\begin{equation*}
\begin{aligned}
\mathbb{E}_{x}^{\mathbb{Q}}[1_{\{T_1  \geq T\}}L(X)_Tf(X_T)] &= \mathbb{E}_{x}[1_{\{\xi_1  \geq T\}}L(X)_Tf(X_T)]\\
&= \mathbb{E}_{x}\left[1_{\{\xi_1  \geq T\}}\exp\left(-\int_0^T\lambda(X_s)ds + \lambda(x)T\right)f(X_T)\right]\\
&= \mathbb{E}_{x}\left[1_{\{\xi_1  \geq T\}}\exp\left(-\int_0^T\lambda(Y_s)ds + \lambda(x)T\right)f(Y_T)\right].
\end{aligned}
\end{equation*}
Recall the construction of the auxiliary process $Z = (Y, \bar{A})$ in $\eqref{eqn_sde_coupled}$ and the $L_2(Z)$ process in $\eqref{L_process}$, one can conclude that 

\begin{equation*}
\begin{aligned}
\mathbb{E}_{x}^{\mathbb{Q}}\left[1_{\{T_1  \geq T\}}L(X)_Tf(X_T)\right] &= e^{-\sigma_A^2T/2}\mathbb{E}_{x}\left[1_{\{\xi_1  \geq T\}}\exp\left(-\bar{A}_T + \lambda(x)T\right)f(Y_T)\right]\\
&= e^{-\sigma_A^2T/2}\mathbb{E}_{x}\left[1_{\{\xi_1  \geq T\}}L_2(Z)_Tf(Y_T)\right].\\
\end{aligned}
\end{equation*}

From Proposition \ref{prop_two_estimators}, there exists an estimator $L_2^{\Theta}(Z^{\pi})_{T}f(Y_T^\pi)$ such that $$\mathbb{E}[L_2^{\Theta}(Z^{\pi})_{T}f(Y_T^\pi)] = \mathbb{E}_{x}[L_2(Z)_Tf(Y_T)].$$ Since $\xi_1$ is exponential distributed random variable with rate $\lambda(x)$ and is (conditionally) independent of $Z$, we can prove that the estimator $1_{\{\xi_1  \geq T\}}\Xi_2(x, T, Z^{\pi})$ defined in \eqref{eqn_xi} will be an unbiased estimator for $\mathbb{E}_{x}^{\mathbb{Q}}[1_{\{T_1 \geq T\}}L(X)_Tf(X_T)]$.

Then we move to the first term inside the expectation of $\eqref{eqn_expansion}$, denote $g(X_{T_1}) = \mathbb{E}_{X_{T_1}}[f(X_T)]$, we have
\begin{equation*}
\begin{aligned}
\mathbb{E}_x^{\mathbb{Q}}[1_{\{T_1 < T\}}L(X)_{T_1}\mathbb{E}_{X_{T_1}}[f(X_T)]] = \mathbb{E}_x^{\mathbb{Q}}[1_{\{T_1 < T\}}L(X)_{T_1}g(X_{T_1})].
\end{aligned}
\end{equation*}
Analogously, we can use the same technique above. Under the event $\{T_1 < T\}$, from $\eqref{construct}$, $\eqref{eqn_sde_coupled}$ and $\eqref{L_process}$ we know that
\begin{equation*}
\begin{aligned}
\mathbb{E}_{x}^{\mathbb{Q}}[1_{\{T_1 < T\}}L(X)_{T_1}g(X_{T_1})] &= \mathbb{E}_{x}[1_{\{\xi_1 < T\}}L(X)_{\xi_1}g(X_{\xi_1})]\\
&= \mathbb{E}_{x}\left[1_{\{\xi_1 < T\}}\exp\left(-\int_0^{\xi_1}\lambda(X_s)ds + \lambda(x)\xi_1\right)\cfrac{\lambda(X_{\xi_1^-})}{\lambda(x)}g(X_{\xi_1})\right]\\
&= \mathbb{E}_{x}\left[1_{\{\xi_1 < T\}}\exp\left(-\int_0^{\xi_1}\lambda(Y_s)ds + \lambda(x)\xi_1\right)\cfrac{\lambda(Y_{\xi_1})}{\lambda(x)}g(Y_{\xi_1} + V_1)\right]\\
&= e^{-\sigma_A^2\xi_1/2}\mathbb{E}_{x}\left[1_{\{\xi_1 < T\}}\exp\left(-\bar{A}_{\xi_1} + \lambda(x)\xi_1\right)\cfrac{\lambda(Y_{\xi_1})}{\lambda(x)}g(Y_{\xi_1}+V_1)\right]\\
&= e^{-\sigma_A^2\xi_1/2}\mathbb{E}_{x}\left[1_{\{\xi_1 < T\}}L_1(Z)_{\xi_1}g(Y_{\xi_1}+V_1)\right].\\
\end{aligned}
\end{equation*}

 From Condition \ref{cdn_blackbox} again, if we know the explicit form of $g$, an unbiased estimator of $\mathbb{E}_{x}\left[1_{\{\xi_1 < T\}}L_1(Z)_{\xi_1}g(Y_{\xi_1}+V_1)\right]$ would be of the form $1_{\{\xi_1 < T\}}L_1^{\Theta}(Z^{\pi})_{\xi_1}g(Y_{\xi_1}^{\pi} + V_1^{\pi})$, where we have $V_1^{\pi}$ because by definition $V_1$ could also be a function of $Y_{\xi_1}^{\pi}$. However, we do not know $g(x)$ and it has to be estimated by generating an unbiased estimator again. This leads to a recursive formulation of the estimator in Theorem \ref{theorem1_recursive}. For the details of the proof we refer the readers to the appendix.

\section{Application with Black-Box Algorithms}\label{sec_parametrix}
In this section we give an example of the black-box algorithm that completes the estimator for the jump-diffusion model. We exemplify our approach on the  ``parametrix'' method developed by  \citet{ballyhiga2015} and  \citet{anderssonhiga2017}. We note that the estimator proposed by \citet{wagner1989} also satisfies Condition \ref{cdn_blackbox} and \ref{cdn_var}. The verification process developed below also applies to this estimator.

Under Assumption \ref{dasm}, for a diffusion model $Y$ satisfying $\eqref{ysde}$, we have
\begin{equation}\label{eqn_unbiased_theta_1}
\begin{aligned}
\mathbb{E}[f(Y_T)] = \mathbb{E}[f(Y_T^{\pi})\Theta_1(Y^{\pi}, T)],
\end{aligned}
\end{equation}
where $f : \mathbb{R}^d \to \mathbb{R}$ is a bounded function, and
\begin{equation}\label{eqn_form_theta_1}
\begin{aligned}
\Theta_1(Y^{\pi}, T) = \frac{1}{\Psi(T-\tau_{N_T})}
  \prod_{k=1}^{N_T}
  \frac{\vartheta_{\tau_k - \tau_{k-1}}
  (Y_{\tau_{k-1}}^{\pi}, Y_{\tau_k}^{\pi})}
  {\psi(\tau_k -\tau_{k-1})},
\end{aligned}
\end{equation}
where $\tau_i$ is the arrival times of the counting process $N$, $\psi$ is the density that we sample $\{\tau_k-\tau_{k-1}\}_{k=1}^\infty$, and $\Psi (t)$ is the survival function
$\Psi (t) = P(\tau_1 > t) = \int_t^\infty \psi(s)ds$. Moreover,
\begin{equation}\label{theta}
\begin{aligned}
\vartheta_t(y_1, y_2) 
&= \frac{1}{2}\sum_{i,j} \vartheta_t^{i,j}(y_1, y_2) 
- \sum_i \rho_t^i(y_1, y_2)\\
\vartheta_t^{i,j}(y_1, y_2) &= \partial_{i,j}^2a^{i,j}(y_2) +
\partial_ja^{i,j}(y_2)h_t^i(y_1, y_2)\\
& +\partial_i a^{i,j}(y_2)h_t^j(y_1, y_2)
 +(a^{i,j}(y_2)-a^{i,j}(y_1))h_t^{i,j}(y_1, y_2)\\
\rho_t^{i}(y_1, y_2) &= \partial_i\mu^i(y_2) + (\mu^i(y_2) - \mu^i(y_1))h_t^i(y_1, y_2)\\
h_t(y_1, y_2) &= H_{ta(y_1)}(y_2-y_1- t \mu(y_1) ),\\
%h_t^{i,j}(y_1, y_2) &= H_{ta(y_1)}^{i, j}(y_1-y_2-\mu(y_1)t)\\
%h_t^{i}(y_1, y_2) &= H_{ta(y_1)}^i(y_2-y_1- t \mu(y_1)),
\end{aligned}
\end{equation}
where $H$ denotes the Hermite polynomials. For any matrix
$M$ we have $1$st-order polynomials defined by $H^i_{M}(x)
= -(M^{-1} x)_i$ and $2$nd-order polynomials define by
$H^{ij}(x) = (M^{-1}x)_i(M^{-1}x)_j -(M^{-1})_{ij}$.

The choice
of a Poisson process for $N$ in \eqref{eqn_form_theta_1} leads to infinite variance. Choices that lead to 
a finite variance are discussed in \citet{anderssonhiga2017}.
One example includes Beta distributed interarrivals on 
$[0,T+\epsilon]$ for a $\epsilon >0$, where $\tau_{i+1} - \tau_i$ become a sequence of i.i.d random variable with density function $f(x) = (1 - \gamma)/(x^{\gamma}(T+\epsilon)^{1-\gamma})$ in the range $[0,T+\epsilon]$. More specifically, in the Beta distributed case we have
\begin{equation*}
\begin{aligned}
\Theta_1(Y^{\pi}, T) = \cfrac{1}{p_{N_T}(\tau_1, \cdots ,\tau_{N_T})}\prod_{j=0}^{N_T-1}\vartheta_{\tau_{j+1}-\tau_j}(Y_{\tau_j}^{\pi}, Y_{\tau_{j+1}}^{\pi}),
\end{aligned}
\end{equation*}
where 
\begin{equation*}
\begin{aligned}	
p_n(s_1, \cdots , s_n) = \left(1-\left(\cfrac{T-s_n}{T+\epsilon}\right)^{1-\gamma}\right)\left(\frac{1-\gamma}{(T+\epsilon)^{1-\gamma}}\right)^n \prod_{i=0}^{n-1}\cfrac{1}{(s_{i+1}-s_i)^{\gamma}}.
\end{aligned}
\end{equation*}

To apply the above estimator to jump-diffusion process, we have to make the notation compatible for the augmented process $Z = (Y, \bar{A}) \in \mathbb{R}^{d+1}$. With the extra dimension, we make slight modifications to $\Theta_1$ defined in \eqref{eqn_form_theta_1}. We define $\Theta_2$ such that
\begin{equation}\label{eqn_form_theta_2}
\begin{aligned}
\Theta_2(Z^{\pi}, T) = \frac{1}{\Psi(T-\tau_{N_T})}
  \prod_{k=1}^{N_T}
  \frac{\theta_{\tau_k - \tau_{k-1}}
  (Z_{\tau_{k-1}}^{\pi}, Z_{\tau_k}^{\pi})}
  {\psi(\tau_k -\tau_{k-1})},
\end{aligned}
\end{equation}
where for $z_1 = (y_1,\bar{a}_1) \in \mathbb{R}^{d+1}$  
and $z_2 = (y_2,\bar{a}_2) \in \mathbb{R}^{d+1}$ we have
\begin{equation}\label{eqn:def_theta}
    \begin{aligned}
    \theta_t(z_1, z_2) = \vartheta_t(y_1, y_2) +
(\lambda(y_2) - \lambda(y_1))
\Big( \frac{\bar{a}_2 - \bar{a}_1 - \lambda(y_1)t}{t\sigma_A^2} \Big),
    \end{aligned}
\end{equation}
which is derived by going through general formulas in \eqref{theta} for the extra dimension.

Next, we need to fix the regularity issue stated in Condition \ref{cdn_blackbox} and \ref{cdn_var}. \citet{anderssonhiga2017} establish the finite variance result for bounded function, and it is not directly applicable to our case. As stated in Proposition \ref{prop_two_estimators}, we have to estimate $L_1(Z)_T$ and $L_2(Z)_Tf(Y_T)$, which involve exponential growth functions. Therefore, we have to make sure the estimator with correction functional $\Theta_2$ defined in \eqref{eqn_form_theta_2} satisfies Condition \ref{cdn_blackbox} and \ref{cdn_var}.

We verify this via two lemmas below. Denote $S_n$ to be the space of $(s_1,\cdots,s_n)$ such that $0 = s_0 < s_1 < \cdots < s_n < T$, and $\gamma$ to be the parameter of the Beta distribution for sampling $\tau_i$'s.

\begin{lemma}\label{lemma_para_var}
Under Assumption \ref{dasm}-\ref{jasm}, let $Z^\pi = (Y^\pi, \bar{A}^\pi)$ be the Euler process in \eqref{eqn_form_theta_2} with starting point $(x, 0)$. For $n \in \mathbb{N}$, $T>0$, and $f: \mathbb{R}^d \to \mathbb{R}$ such that 
$|f(\bm{y})| \leq e^{c_1||\bm{y}||_1 + c_2}$, there exist $M(T, p)$ and $C_T$ such that
\begin{equation*}
\begin{aligned}	
\mathbb{E}\left[\left|e^{-\bar{A}_T^{\pi}}e^{c_1||Y_T^\pi||_1 + c_2}\Theta_{2}(Z^\pi, T)\right|^p\right]\leq M(T, p)\prod_{i=1}^d \cosh(c_1px_{0}^{(i)}),
\end{aligned}
\end{equation*}
where $$\bar{M}(T,p) :=  C_T\exp(((2c_1^2a_2d+\sigma_A^2)p^2+\lambda_2p)T),$$
$$M(T,p) :=
\bar{M}(T,p)\int_{S^n}\frac{1}{p_n(s_1, \cdots , s_n)^{p-1}}\prod_{j=1}^{n}(s_j - s_{j-1})^{-p/2}ds.$$
\end{lemma}

Lemma \ref{lemma_para_var} characterize the moment condition for the functional we are interested in. The first moment $p=1$ being finite ensures the existence of expectation, and the second moment $p=2$ being finite ensures the finite variance property. From Proposition 7.3 in \citet{anderssonhiga2017} and the lemma above, we know that the $p$-th moment will be finite if $p(1/2 - \gamma) < 1-\gamma$. For example, a safe choice ensuring the first and second moment being finite is to choose  $\gamma \in (0,1/2)$. Therefore, by recalling that $\cosh(x) = (e^x + e^{-x})/2$, we have $\prod_{i=1}^d \cosh(c_1px_{0}^{(i)}) \leq \exp\left(c_1p||x_0||_1\right)$, thereby verifying Condition \ref{cdn_var}. The reason we use hyperbolic function is that it  features better analytical trackability and will be used in the later proof.

We then state the unbiasedness result for our estimator, which will ensure Condition \ref{cdn_blackbox}.
\begin{lemma} \label{lem_para_expectation}
Under Assumption \ref{dasm}-\ref{jasm}, let $Z_t = (Y_t, \bar{A}_t)$ as defined as \eqref{eqn_sde_coupled}, and $Z^\pi = (Y^\pi, \bar{A}^\pi)$ be the Euler process in \eqref{eqn_form_theta_2} with starting point $(x, 0)$. If we sample the arrival time $\xi_i = \tau_i - \tau_{i-1}$ from Beta distribution with parameter $\gamma \in (0,1/2)$, for $f : \mathbb{R}^d \to \mathbb{R}$ with exponential growth, the following representation holds

\begin{equation*}
\begin{aligned}	
\mathbb{E}\left[e^{-\bar{A}_T} f(Y_T)\right] = \mathbb{E}\left[e^{-\bar{A}_T^{\pi}}f(Y_T^{\pi})\Theta_2(Z^\pi, T)\right].
\end{aligned}
\end{equation*}
\end{lemma}

With  Condition \ref{cdn_blackbox} and \ref{cdn_var} verified from Lemma \ref{lemma_para_var} and \ref{lem_para_expectation}, we know that $\Theta_2$ is the right correction functional to accommodate jump-diffusions.  

\section{Numerical Experiments}

In this section we conduct numerical experiments. Firstly, we test the performance and robustness of our unbiased jump-diffusion estimator. We conduct the numerical experiment under two environments: the first one satisfying Assumption \ref{dasm}-\ref{jasm}, and the second one violates those assumptions. It turns out that our algorithm is unbiased under both environments.

Secondly, we show the efficiency of our algorithm compared to the standard Euler scheme. Due to the fact that we sample much fewer grid size than the standard Euler algorithm, our algorithm is much faster per trail. The down side is a higher variance, but from experiments we can see the faster speed outweighs the higher variance, thereby making the jump-diffusion estimator more efficient than the standard Euler. As shown in our results, the jump-diffusion estimator has a higher convergence rate and a significant improvement in efficiency compare to Euler.

Lastly, we discuss the effects of different parameters. Recall that we have $3$ parameters that need to be chosen by design: the diffusion multiplier $\sigma_A$ for the auxiliary process $\bar{A}$, $\gamma$ controlling the Beta distribution, and $\epsilon$ for sampling the grid size between jump times. We found that from our experiment, choosing appropriate $\sigma_A$ and $\epsilon$ is essential to the efficiency and accuracy of our algorithm.

\subsection{Performance on different models}

Firstly, we test our result in an environment that satisfy Assumption \ref{dasm} and \ref{jasm}, and we choose the naive model where all functions are bounded (except for $f_2$):
\begin{equation}\label{model_1}
    \begin{aligned}
    dX_t^{(1)} &= (\mu_1 - \mu_2\sin(X_t^{(1)}))dt + \sqrt{\sigma_1 + \sigma_2\sin(X_t^{(1)})} dW_t^{(1)}\\
    dX_t^{(2)} &= (\mu_1 - \mu_2\cos(X_t^{(2)}))dt + \sqrt{\sigma_1 + \sigma_2\sin(X_t^{(2)})} dW_t^{(2)}\\
\lambda(x) &= \lambda_1 + \lambda_2\sin(\lambda_3X_t^{(1)} + \lambda_4X_t^{(2)})\\ f_1(x) &= 1_{\{X_T^{(1)} + X_T^{(2)} > k\}} \,\,\,\,\, f_2(x) = (X_T^{(1)} + X_T^{(2)} - k)_+,
    \end{aligned}
\end{equation}
where $\mu_1 = 0.4$, $\mu_2 = 0.2$, $\sigma_1 = 1$, $\sigma_2 = 0.2$, $\lambda_1 = 0.3$, $\lambda_2 = \lambda_3 = \lambda_4 = 0.2$ and $k = 1.8$. The reason for choosing $f_1$ and $f_2$ is that they include a wide family of payoffs of practical interest. 

We test the ``parametrix'' jump-diffusion estimator described in Section \ref{sec_parametrix}, and compare it with the Euler
method developed in \cite{shkolnik2021}. For implementation of the Euler method, we choose the well-known allocation rule \cite{duffie1995} that is asymptotically efficient. A nearly exact expectation is computed with a very large number of Monte Carlo trials.

Tables 1 summarize major performance in our
experiments on estimating $\mathbb{E}[f_1(X_T)]$. We find both algorithm converges to the true expected value. Moreover, we compare the efficiency between two algorithms in Figure \ref{fig:efficiency}. An algorithm is more efficient if it generates smaller confidence interval given the same run time. We find that the parametrix jump-diffusion estimator is more efficient than the Euler approximation.

\begin{table}[htb]
\centering
\begin{tabular}{c|cccc|ccccc}
\toprule
\multicolumn{8}{ c }{Parametrix \hspace{1.8in} Euler} \\
\midrule
Type & $M$ & Error & Var & $99\%$ CI 
 & $M$ & $p$ & Error & Var & $99\%$ CI 
\\ \midrule
$f_1$ & $5\times10^4$ & 0.00166 & 2.48 &  0.018
	& $4\times 10^3$ &  $2\times10^2$ & 0.00659 &0.218 & 0.019   \\
$f_1$ & $5\times10^5$ & 0.0018 & 2.5 &  0.0057
	& $16\times 10^3$ & $4\times10^2$ & 0.003778 &0.219 & 0.009   \\
$f_1$ & $5\times10^6$ &  0.0002 & 2.48 &  0.0018
	& $64\times 10^3$ & $8\times10^2$ & 0.00153 &0.220 & 0.005  \\
	\midrule
$f_2$ & $5\times10^4$ & 0.0025 & 2.35 &  0.0187
	& $4\times 10^3$ &  $2\times10^2$ & 0.0041 & 0.21273 & 0.0188   \\
$f_2$ & $5\times10^5$ & 0.00063 & 2.47 &  0.0059
	& $16\times 10^3$ & $4\times10^2$ & 0.00298 & 0.21737 & 0.0094   \\
$f_2$ & $5\times10^6$ &  0.00088 & 2.67 &  0.0018
	& $64\times 10^3$ & $8\times10^2$ & 0.00115 & 0.21655& 0.005  \\
\bottomrule
\end{tabular}
\label{tbl:int}
\caption{Estimation of $\mathbb{E}[f_1(X_T)]$ and $\mathbb{E}[f_2(X_T)]$ for model $\eqref{model_1}$
with the parametrix and Euler methods.
\tq{Error} reports the absolute 
value between the (nearly) exact value and the 
Monte Calro estimate based on $M$ trials.
Normal confidence intervals (CI) accompany each estimate.
}
\end{table}

\begin{table}[htb]
\centering
\begin{tabular}{c|cccc|ccccc}
\toprule
\multicolumn{8}{ c }{Parametrix \hspace{1.8in} Euler} \\
\midrule
Type & $M$ & Error & Var & $99\%$ CI 
 & $M$ & $p$ & Error & Var & $99\%$ CI 
\\ \midrule
$f_1$ & $5\times10^4$ & 0.0175 & 3.65 & 0.0209
	& $4\times 10^3$ &  $2\times10^2$ & 0.0042 & 0.236 & 0.0198   \\
$f_1$ & $5\times10^5$ & 0.00159 & 3.35 & 0.0066
	& $16\times 10^3$ & $4\times10^2$ & 0.0086 & 0.233 & 0.0099   \\
$f_1$ & $5\times10^6$ &  0.001 & 3.36 & 0.0021
	& $64\times 10^3$ & $8\times10^2$ & 0.00069 & 0.235 & 0.0049 \\
	\midrule
$f_2$ & $5\times10^4$ & 0.001 & 3.90 &  0.0236
	& $4\times 10^3$ &  $2\times10^2$ & 0.0066 & 0.256 & 0.021   \\
$f_2$ & $5\times10^5$ & 0.0007 & 4.95 &  0.0075
	& $16\times 10^3$ & $4\times10^2$ &  0.0059 & 0.275 & 0.0104   \\
$f_2$ & $5\times10^6$ &  0.0005 & 4.24 &  0.0023
	& $64\times 10^3$ & $8\times10^2$ & 0.0015 & 0.269 & 0.005  \\
\bottomrule
\end{tabular}
\label{tbl:int1}
\caption{Estimation of $\mathbb{E}[f_1(X_T)]$ and $\mathbb{E}[f_2(X_T)]$ for model $\eqref{model_2}$
with the parametrix and Euler methods.
\tq{Error} reports the absolute 
value between the (nearly) exact value and the 
Monte Calro estimate based on $M$ trials.
Normal confidence intervals (CI) accompany each estimate.
}
\end{table}

Since the assumptions we put on our jump-diffusion estimator are quite restrictive, we apply it to the affine jump-diffusion model \cite{duffie2000} where all the coefficients in the model have an affine form
\begin{equation}\label{model_2}
    \begin{aligned}
    dX_t^{(1)} &= (\mu_1 - \mu_2X_t^{(1)})dt + \sqrt{\sigma_1 + \sigma_2X_t^{(1)}} dW_t^{(1)}\\
    dX_t^{(2)} &= (\mu_3 - \mu_4X_t^{(2)})dt + \sqrt{\sigma_1 + \sigma_2X_t^{(2)}} dW_t^{(2)}\\
\lambda(x) &= \lambda_1 + \lambda_2X_t^{(1)} + \lambda_3X_t^{(2)}\\ f_1(x) &= 1_{\{X_T^{(1)} + X_T^{(2)} > k\}} \,\,\,\,\, f_2(x) = (X_T^{(1)} + X_T^{(2)} - k)_+
    \end{aligned}
\end{equation}
with $\mu_1 = 0.6$, $\mu_2 = 0.1$, $\mu_3 = 0.5$, $\mu_4 = 0.2$, $\sigma_1 = 1$, $\sigma_2 = 0.2$, $\lambda_1 = 0.3$, $\lambda_2 = \lambda_3 = 0.04$ and $k = 1.8$. From table 2 and Figure \ref{fig:efficiency} we see the performance is still consistent with the previous result regardless of the fact that the condition has been violated.

\begin{figure}[h]
\centering
\begin{subfigure}{.4\textwidth}
  \centering
  \includegraphics[width=1\linewidth]{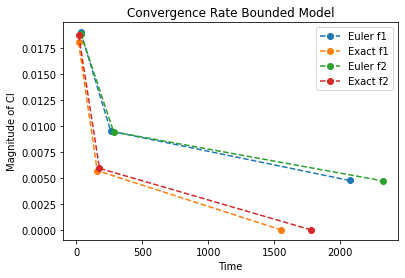}
  \caption{Results on model $\eqref{model_1}$}
  \label{fig:bdd}
\end{subfigure}
\begin{subfigure}{.4\textwidth}
  \centering
  \includegraphics[width=1\linewidth]{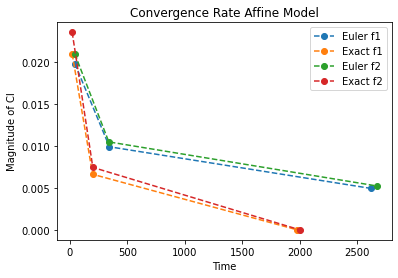}
  \caption{Results on model $\eqref{model_2}$}
  \label{fig:affine}
\end{subfigure}

\caption{
    An algorithm is more efficient if it can generate smaller magnitude of confidence interval given the same amount of time. We find that the more accurate we want the estimation be, the more efficient the parametrix method is compared to the Euler approximation. Moreover, such advantage in efficiency is consistent on different models even with assumption being violated.
}
\label{fig:efficiency}
\end{figure}

\subsection{Sensitivity on Parameters}
In this subsection we test on the behavior of the algorithm based on different parameters. More specifically, we are interested in the choice of $\sigma_A$, the diffusion parameter of the auxiliary process in $\eqref{eqn_sde_coupled}$, $\epsilon$, the parameter controlling the length of support for sampling the grid, and $\gamma$, the parameter for beta distribution. 

We test the result based on the model $\eqref{model_1}$ with payoff function $f_1$. The true value is generated by running the exact algorithm for $10^9$ times.

Generally speaking, the optimal parameter setup would be those having the lowest running time and smallest sample variance. This is because the sample variance directly affect the magnitude of our confidence interval.

For $\sigma_A$, from the description of our algorithm, it is clear that the change of $\sigma_A$ will not affect the running time. Therefore, we want to observe the effect of different $\sigma_A$ on the error and variance. By setting $\sigma_A$ to $0.01, 0.1, 0.5, 1$ and $5$, and running $2\times 10^6$ independent trails, we found that $\sigma_A = 0.01$ and $5$ will have very large error and variance, compared to those much smaller variance for $\sigma_A = 0.1$ $0.5$, and $1$. Please refer to Figure \ref{fig:sigma} for more details.

For $\gamma$, the results show that the running time for different $\gamma$ is roughly the same, henceforth we are not reporting the difference of the running time. As shown is Figure \ref{fig:gamma} There is a minor difference for the performance in terms of error and variance, but the difference is so small that one can view them as random noises generated by the algorithm. Our conclusion is that the performance of our algorithm is quite robust on different $\gamma$.

Lastly for $\epsilon$, since $\epsilon$ is controlling the magnitude of the support of the grid points, smaller $\epsilon$ will results in more grid points, thereby making the algorithm slower. Our experiment (Figure \ref{fig:eps2}) verify this finding. Moreover, we found varying the $\epsilon$ also changes the error and variance a lot. To make a consistent benchmark, we report our ``time'' in benchmark as the time required for the algorithm to have confidence interval of magnitude $10^{-4}$. We found that if $\epsilon < 0.5$, the error and variance will be very large, not to mention the massive running time. However, although setting $\epsilon$ larger will result in faster running time, we found that $\epsilon$ being too large will also affect the error. Therefore, as shown in Figure \ref{fig:eps1}, setting $\epsilon \in [0.5, 5]$ will result in good balance between accuracy and efficiency.

\begin{figure}[h]
\centering
\begin{subfigure}{.4\textwidth}
  \centering
  \includegraphics[width=1\linewidth]{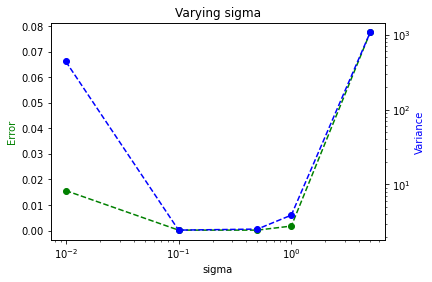}
  \caption{Test on $\sigma_A$}
  \label{fig:sigma}
\end{subfigure}
\begin{subfigure}{.4\textwidth}
  \centering
  \includegraphics[width=1\linewidth]{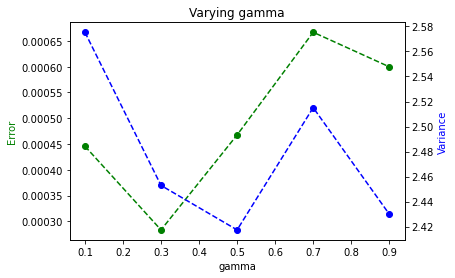}
  \caption{Test on $\gamma$}
  \label{fig:gamma}
\end{subfigure}
\begin{subfigure}{.4\textwidth}
  \centering
  \includegraphics[width=1\linewidth]{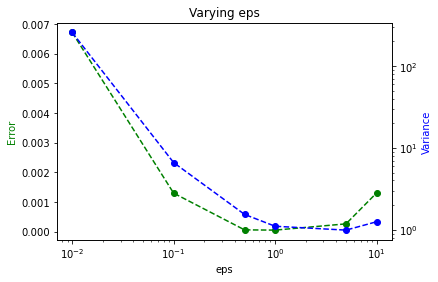}
  \caption{Test on $\epsilon$}
  \label{fig:eps1}
\end{subfigure}
\begin{subfigure}{.4\textwidth}
  \centering
  \includegraphics[width=1\linewidth]{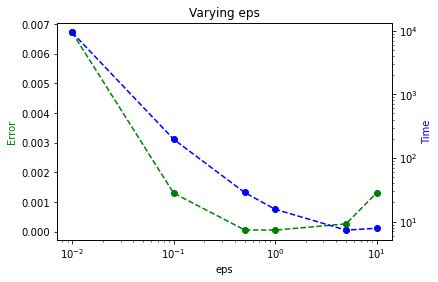}
  \caption{Test on $\epsilon$}
  \label{fig:eps2}
\end{subfigure}

\caption{
}
\end{figure}

%\section{Conclusion}
%We extend a class of unbiased simulation estimators for diffusion models such that they become compatible with jump-diffusions models. Our
%theoretical results establish the unbiasedness and finite
%variance properties of the estimators. Numerical examples
%illustrates the robustness of the algorithm beyond standard regularity conditions. We also show the advantage of our algorithm over
%Euler discretization scheme in terms of efficiency.

\bibliographystyle{jmr}%{model1-num-names}
\bibliography{biblio_MC.bib}

@article{levi1907,
  title={Sulle equazioni lineari totalmente ellittiche alle derivate parziali},
  author={Levi, Eugenio Elia},
  journal={Rendiconti del Circolo Matematico di Palermo (1884-1940)},
  volume={24},
  number={1},
  pages={275--317},
  year={1907},
  publisher={Springer}
}

@inproceedings{chen2019,
  title={Unbiased simulation estimators for jump-diffusions},
  author={Chen, Guanting and Shkolnik, Alexander and Giesecke, Kay},
  booktitle={2019 Winter Simulation Conference (WSC)},
  pages={890--901},
  year={2019},
  organization={IEEE}
}

@inproceedings{chen2020,
  title={Unbiased simulation estimators for path integrals of diffusions},
  author={Chen, Guanting and Shkolnik, Alex and Giesecke, Kay},
  booktitle={2020 Winter Simulation Conference (WSC)},
  pages={277--288},
  year={2020},
  organization={IEEE}
}

@article{ballyhiga2015,
  title={A probabilistic interpretation of the parametrix method},
  author={Bally, Vlad and Kohatsu-Higa, Arturo},
  journal={The Annals of Applied Probability},
  volume={25},
  number={6},
  pages={3095-3138},
  year={2015},
  publisher={Institute of Mathematical Statistics}
}

@article{menozzi2021,
  title={Density and gradient estimates for non degenerate Brownian SDEs with unbounded measurable drift},
  author={Menozzi, S and Pesce, A and Zhang, X},
  journal={Journal of Differential Equations},
  volume={272},
  pages={330--369},
  year={2021},
  publisher={Elsevier}
}

@article{duffie1995,
  title={Efficient Monte Carlo simulation of security prices},
  author={Duffie, Darrell and Glynn, Peter},
  journal={The Annals of Applied Probability},
  pages={897--905},
  year={1995},
  publisher={JSTOR}
}

@article{anderssonhiga2017,
  title={Unbiased simulation of stochastic differential equations using parametrix expansions},
  author={Andersson, Patrik and Kohatsu-Higa, Arturo},
  journal={Bernoulli},
  volume={23},
  number={3},
  pages={2028-2057},
  year={2017},
  publisher={Bernoulli Society for Mathematical Statistics and Probability}
}

@article{glynn2015,
  title={Unbiased Estimation with Square Root Convergence for
SDE Models},
  author={Glynn, Peter and Rhee, Chang-Han},
  journal={Operations Research},
  volume={63},
  number={5},
  pages={1026-1043},
  year={2015},
  publisher={INFORMS}
}

@article{blanchet2017,
  title={$\varepsilon$-Strong simulation for multidimensional stochastic differential equations via rough path analysis},
  author={Blanchet, Jose and Chen, Xinyun and Dong, Jing},
  journal={The Annals of Applied Probability},
  volume={27},
  number={1},
  pages={275--336},
  year={2017},
  publisher={Institute of Mathematical Statistics}
}

@article{pollock2016exact,
  title={On the exact and $\varepsilon$-strong simulation of (jump) diffusions},
  author={Pollock, Murray and Johansen, Adam M and Roberts, Gareth O},
  journal={Bernoulli},
  volume={22},
  number={2},
  pages={794--856},
  year={2016},
  publisher={Bernoulli Society for Mathematical Statistics and Probability}
}

@article{doumbia2017,
  title={Unbiased Monte Carlo estimate of stochastic differential equations expectations},
  author={Doumbia, Mahamadou and Oudjane, Nadia and Warin, Xavier},
  journal={ESAIM: Probability and Statistics},
  volume={21},
  pages={56--87},
  year={2017},
  publisher={EDP Sciences}
}

@inproceedings{agarwal2017,
  title={Finite variance unbiased estimation of stochastic differential equations},
  author={Agarwal, Ankush and Gobet, Emmanuel},
  booktitle={2017 Winter Simulation Conference (WSC)},
  pages={1950--1961},
  year={2017},
  organization={IEEE}
}

@book{platen2010,
  title={Numerical solution of stochastic differential equations with jumps in finance},
  author={Platen, Eckhard and Bruti-Liberati, Nicola},
  volume={64},
  year={2010},
  publisher={Springer Science \& Business Media}
}

@incollection{protter2005,
  title={Stochastic differential equations},
  author={Protter, Philip E},
  booktitle={Stochastic integration and differential equations},
  pages={249--361},
  year={2005},
  publisher={Springer}
}

@article{beskos2005,
  title={Exact simulation of diffusions},
  author={Beskos, Alexandros and Roberts, Gareth O},
  journal={The Annals of Applied Probability},
  volume={15},
  number={4},
  pages={2422-2444},
  year={2005},
  publisher={Institute of Mathematical Statistics}
}

@article{beskos2006,
  title={Retrospective Exact Simulation of Diffusion Sample Paths with Applications},
  author={Beskos, Alexandros and Papaspiliopoulos, Omiros and Roberts, Gareth O},
  journal={Bernoulli},
  volume={12},
  number={6},
  pages={1077-1098},
  year={2006},
  publisher={Bernoulli Society for Mathematical Statistics and Probability}
}

@article{wagner1989,
  title={Unbiased Monte Carlo estimators for functionals of weak solutions of stochastic differential equations},
  author={Wagner, Wolfgang},
  journal={Stochastics An International Journal of Probability and Stochastic Processes},
  volume={28},
  number={1},
  pages={1-20},
  year={1989},
  publisher={Taylor & Francis}
}

@article{chen2013,
  title={Localization and Exact Simulation of Brownian Motion-
Driven Stochastic Differential Equations},
  author={Chen, Nan and Huang, Zhengyu},
  journal={Mathematics of Operations Research},
  volume={38},
  number={3},
  pages={591-616},
  year={2013},
  publisher={INFORMS}
}

@article{casella2011,
  title={Exact Simulation of Jump-Diffusion Processes with Monte Carlo Applications},
  author={Casella, Bruno and Roberts, Gareth O},
  journal={Methodology and Computing in Applied Probability},
  volume={13},
  number={3},
  pages={449-473},
  year={2011},
  publisher={Springer-Verlag}
}

@article{giesecke2013,
  title={Exact Sampling of Jump Diffusions},
  author={Giesecke, Kay and Smelov, Dmitry},
  journal={Operations Research},
  volume={61},
  number={4},
  pages={894-907},
  year={2013},
  publisher={INFORMS}
}

@article{shkolnik2021,
  title={Numerical solution of jump-diffusion SDEs},
  author={Shkolnik, Alexander and Giesecke, Kay and Teng, Gerald and Wei, Yexiang},
  journal={Operations Research, forthcoming},
  year={2021}
}

@book{tavella2000,
  title={Pricing financial instruments: The finite difference method},
  author={Tavella, Domingo and Randall, Curt},
  volume={13},
  year={2000},
  publisher={John Wiley \& Sons}
}

@book{cinlar2011,
  title={Probability and stochastics},
  author={Cinlar, Erhan},
  volume={261},
  year={2011},
  publisher={Springer}
}

@article{labordere2017,
  title={Unbiased simulation of stochastic differential equations},
  author={Henry-Labordère, Pierre and Tan, Xiaolu and Touzi, Nizar},
  journal={The Annals of Applied Probability},
  volume={27},
  number={6},
  pages={3305-3341},
  year={2017},
  publisher={Institute of Mathematical Statistics}
}

@article{duffie2000,
  title={Transform analysis and asset pric-
ing for affine jump-diffusions},
  author={Duffie, Darrell and Pan, Jun and Singleton, Kenneth},
  journal={Econometrica},
  volume={68},
  number={6},
  pages={1343-1376},
  year={2000},
  publisher={The Econometric Society}
}

@article{goncalves2014,
  title={Exact Simulation Problems for Jump-Diffusions},
  author={Gonçalves, Flávio B and Roberts, Gareth O },
  journal={Methodology and Computing in Applied Probability},
  volume={16},
  number={4},
  pages={907-930},
  year={2014},
  publisher={Springer}
}

@book{kloeden1999,
  title={Numerical Solution of Stochastic Differential Equations},
  author={Kloeden, Peter E. and Platen, Eckhard},
  year={1999},
  publisher={Springer}
}

@article{bias2021,
  title={Reducing Bias in Event Time Simulations via Measure Changes},
  author={Giesecke, Kay and Shkolnik, Alexander},
  year = {2021},
  journal={Mathmatics of Operations Research, forthcoming}
}

@article{blanchet2020,
  title={Exact simulation for multivariate It{\^o} diffusions},
  author={Blanchet, Jose and Zhang, Fan},
  journal={Advances in Applied Probability},
  volume={52},
  number={4},
  pages={1003--1034},
  year={2020},
  publisher={Cambridge University Press}
}

\newpage

\appendix

\section{Proofs: Unbiasedness}
Before proving the unbiased result, it is helpful to show that
\begin{equation}\label{verify_l2}
\begin{aligned} 
\mathbb{E}[\Xi_2(x, T, Z^{\pi})1_{\{\xi_1 > T\}}] = \mathbb{E}_{x}^{\mathbb{Q}}[L(X)_Tf(X_T)1_{\{T_1 \geq T\}}].
\end{aligned}
\end{equation}
By conditioning, we have
\begin{equation*}
\begin{aligned}
\mathbb{E}_{x}^{\mathbb{Q}}[L(X)_Tf(X_T)1_{\{T_1 \geq T\}}] &\overset{\mathrm{(a)}}{=} \mathbb{E}_{x}^{\mathbb{Q}}[\mathbb{E}_{\xi_1}^{\mathbb{Q}}[L(X)_Tf(X_T)1_{\{\xi_1 \geq T\}}]]\\
&\overset{\mathrm{(b)}}{=} \mathbb{E}_{x}^{\mathbb{Q}}[\mathbb{E}_{\xi_1}[L(X)_Tf(X_T)1_{\{\xi_1 \geq T\}}]]\\
&\overset{\mathrm{(c)}}{=} \mathbb{E}_{x}^{\mathbb{Q}}[1_{\{\xi_1 \geq T\}}\mathbb{E}_{\xi_1}[L_2(Y)_Tf(Y_T)]]\\
&\overset{\mathrm{(d)}}{=} \mathbb{E}_{x}[1_{\{\xi_1 \geq T\}}\mathbb{E}[L_2(Y)_Tf(Y_T)]],
\end{aligned}
\end{equation*}
where $(a)$ comes from tower property and the fact that $T_1 = \xi_1$ a.s. under $\mathbb{Q}$. $(b)$ comes from the fact that the law of Brownian Motion $W$ under measure $\mathbb{Q}$ and $\mathbb{P}$ are the same. $(c)$ holds because of $\eqref{construct}$ and $\eqref{L_process}$. $(d)$ holds because $\xi_1$ is independent of $Y$, and $\xi_1$ and $Y$ has the same distribution under $\mathbb{Q}$ and $\mathbb{P}$. Next, from $\eqref{eqn_sde_coupled}$, since $\bar{W}$ is independent from other processes, we know that for any function $g$, 
\begin{equation*}
\begin{aligned}
\mathbb{E}[\exp{\left(-\bar{A}_t\right)}g(Y_T)] = \exp(\sigma_A^2T/2)\mathbb{E}\left[\exp{\left(-\int_0^t \lambda(Y_s)ds\right)}g(Y_T)\right].
\end{aligned}
\end{equation*}
This along with Condition \ref{cdn_blackbox} yields
\begin{equation*}
\begin{aligned}
\mathbb{E}[\exp\left({-\sigma_A^2T/2}\right)L_2^{\Theta}(Z^{\pi})_{T}f(Y_T^{\pi})] = \mathbb{E}[L_2(Y)_Tf(Y_T)],
\end{aligned}
\end{equation*} 
which proves $\eqref{verify_l2}$.

We begin the proof of unbiasedness by induction, and before that we need some notations. Denote $N_t^{\mathbb{Q}}$ to be the number of jumps before time $t$ under measure $\mathbb{Q}$, and $T_i^{\mathbb{Q}}$ to be the $i$-th jump time of the process $X$ under measure $\mathbb{Q}$. Define event $A_{n, T} = \{N_T < n\}$. Moreover, denote $T_i^\xi$ to be $\sum_{j=1}^i\xi_j$, $N_T^\xi = \max\{n:T_n^\xi \leq T\}$, and event $A_{n,T}^\xi = \{N_T^\xi < n\}$.
Notice that under measure $\mathbb{Q}$, $T_i^\xi = T_i^\mathbb{Q}$ a.s. Also under measure $\mathbb{Q}$ we have $A_{n, T}^{\mathbb{Q}} = \{N_T^{\mathbb{Q}} < n\} = A_{n, T}^{\xi}$. If we can show
\begin{equation}\label{assum_thm_1}
\begin{aligned}
\mathbb{E}_{x}[U(x, T)1_{A_{n, T}^\xi}] = \mathbb{E}_{x}[f(X_T) 1_{A_{n, T}}],
\end{aligned}
\end{equation}
then from $\mathbb{E}[|U(x, T)|] < \infty$ (which follows from Theorem 2), Dominated Convergence Theorem will guarantee the unbiasedness result. For any $T > 0$, we start with the base case, on event $A_{1, T}^\xi = \{\xi_1 > T\} = \{T_1^{\mathbb{Q}} > T\}$, from $\eqref{verify_l2}$ we know
\begin{align}\label{ap_thm_unbiased_base}
    \mathbb{E}_{x}[U(x, T)1_{A_{1, T}^\xi}] = \mathbb{E}_{x}[\Xi_2(x, T, Z^{\pi})1_{\{\xi_1 > T\}}]  = \mathbb{E}_{x}^{\mathbb{Q}}[L(X)_Tf(X_T)1_{\{T_1 > T\}}] = \mathbb{E}_{x}[f(X_T) 1_{\{T_1 > T\}}].
\end{align}
Then we begin the induction hypothesis. Suppose that $\eqref{assum_thm_1}$ is true for some $n > 1$, 
we want to show 
\begin{equation}\label{assum_n+1_thm_1}
\begin{aligned}
\mathbb{E}_{x}[U(x, T)1_{A_{n+1, T}^\xi}] = \mathbb{E}_{x}[f(X_T) 1_{A_{n+1, T}}].
\end{aligned}
\end{equation}
Decomposing into two parts again and applying $\eqref{ap_thm_unbiased_base}$ yields
\begin{equation*}
\begin{aligned}
\mathbb{E}_{x}[U(x, T)1_{A_{n+1, T}^\xi}] &= \mathbb{E}_{x}[1_{\{\xi_1<T\}}U(x, T)1_{A_{n+1, T}^\xi}] + \mathbb{E}_{x}[1_{\{\xi_1\geq T\}}U(x, T)1_{A_{n+1, T}^\xi}]\\
&= \mathbb{E}_{x}[1_{\{\xi_1<T\}}U(x, T)1_{A_{n+1, T}^\xi}] + \mathbb{E}_{x}[f(X_T)1_{A_{1, T}}].
\end{aligned}
\end{equation*}
Therefore, it suffices to show that
\begin{equation*}
\begin{aligned}
\mathbb{E}_{x}[1_{\{\xi_1<T\}}U(x, T)1_{A_{n+1, T}^\xi}] = \mathbb{E}_{x}[f(X_T)1_{A_{n+1, T} \cap A_{1, T}^c}].
\end{aligned}
\end{equation*}
Recall the notation $X_{\xi_1}^\pi = Y_{\xi_1}^{\pi} + V_1^\pi$, and define event $A_{n, T,\xi_1}^{\xi}:= \{N_{T}^{\xi} - N_{\xi_1}^{\xi}<n\}$. Next,  
\begin{equation}\label{eqn:unbiased_before_last}
\begin{aligned}
&\hspace{5.5mm}\mathbb{E}_{x}[1_{\{\xi_1<T\}}U(x, T)1_{A_{n+1, T}^\xi}]\\
 &= \mathbb{E}_{x}[1_{\{\xi_1< T\}}\Xi_1(x, \xi_1, Z^{\pi})U(Y_{\xi_1}^{\pi} + V_1^\pi, T-\xi_1)1_{A_{n+1, T}^\xi}]\\
&=\mathbb{E}_{x}[1_{\{\xi_1< T\}}\mathbb{E}_{\xi_1}[\Xi_1(x, \xi_1, Z^{\pi})\mathbb{E}_{X_{\xi_1}^{\pi}, \xi_1}[U(X_{\xi_1}^{\pi}, T-\xi_1)1_{A_{n, T, \xi_1}^{\xi}}]]].
\end{aligned}
\end{equation}
In order to apply strong Markov property and the induction hypothesis, we need to introduce some new notations. Denote the process $\bar{X}$ as the jump-diffusion and diffusion process having the same dynamics as in \eqref{s1_JD_sde}, but with a starting point $\bar{X}_0 = X_{\xi_1}^\pi = Y_{\xi_1}^\pi + V_1^\pi$, where $\bar{V}_1^\pi = h(Y_{\xi_1}^\pi, R_1)$. Moreover, we denote $\bar{N}$ as the corresponding counting process related to $\bar{X}$, and denote the event $\bar{A}_{n,T} = \{\bar{N}_T < n\}$. From induction hypothesis
we know that 
\begin{equation*}
\begin{aligned}
\mathbb{E}_{X_{\xi_1}^{\pi}, \xi_1}[U(X_{\xi_1}^{\pi}, T-\xi_1)1_{A_{n, T, \xi_1}^{\xi}}] = \mathbb{E}[f(\bar{X}_{T-\xi_1})1_{\bar{A}_{n, T-\xi_1}} |\, \bar{X}_0 = X_{\xi_1}^{\pi}].
\end{aligned}
\end{equation*}
From the relation that
$\bar{X}_0 = X_{\xi_1}^\pi = Y_{\xi_1}^\pi + V_1^\pi$, for notational simplicity we denote

\begin{equation*}
\begin{aligned}
\zeta_x(x,\xi) &= \mathbb{E}\left.\left[f(\bar{X}_{T-\xi})1_{\bar{A}_{n, T - \xi}}\right|\bar{X}_0 = x\right]\\
\zeta_y(\bm{y},\xi) &= \mathbb{E}_{\bm{y}}[\mathbb{E}_{\bar{X}_0}[f(\bar{X}_{T-\xi})1_{\bar{A}_{n, T - \xi}}]] = \mathbb{E}_{\bm{y}}[\zeta_x(\bar{X}_0,\xi)],
\end{aligned}
\end{equation*}
where $\bar{X}_0 = \bm{y} + h(\bm{y}, R_1)$ and $R_1$ is an independent random variable with law $\nu$. Then, to finish the proof, from induction hypothesis and $\eqref{eqn:unbiased_before_last}$ we have
\begin{equation*}
\begin{aligned}
&\hspace{5.5mm}\mathbb{E}_{x}[1_{\{\xi_1<T\}}U(x, T)1_{A_{n+1, T}^\xi}]\\
&=\mathbb{E}_{x}[1_{\{\xi_1< T\}}\mathbb{E}_{\xi_1}[\Xi_1(x, \xi_1,  Z^{\pi})\mathbb{E}_{X_{\xi_1}^{\pi}, \xi_1}[f(\bar{X}_{T-\xi_1})1_{\bar{A}_{n, T - \xi_1}}]]]\\
&=\mathbb{E}_{x}[1_{\{\xi_1< T\}}\mathbb{E}_{\xi_1}[\Xi_1(x, \xi_1,  Z^{\pi})\mathbb{E}_{Y^{\pi}_{\xi_1}}[\mathbb{E}_{X_{\xi_1}^{\pi}, \xi_1}[f(\bar{X}_{T-\xi_1})1_{\bar{A}_{n, T - \xi_1}}]]]]\\
&=\mathbb{E}_{x}[1_{\{\xi_1< T\}}\mathbb{E}_{\xi_1}[\Xi_1(x, \xi_1, Z^{\pi})\zeta_y(Y_{\xi_1}^{\pi},\xi_1)]]\\
&=\mathbb{E}_{x}\left[1_{\{\xi_1< T\}}\mathbb{E}_{\xi_1}[\exp\left({-\sigma_A^2\xi_1/2}\right)L_1^{\Theta}(Z^{\pi})_{\xi_1}\zeta_y(Y^{\pi}_{\xi_1},\xi_1)]\right]\\
&\overset{\mathrm{(a)}}{=}\mathbb{E}_{x}\left[1_{\{\xi_1< T\}}\exp\left(-\int_0^{\xi_1}Y_sds+\xi_1\lambda(x)\right)\frac{\lambda(Y_{\xi_1})}{\lambda(x)}\zeta_y(Y_{\xi_1},\xi_1)\right]\\
&\overset{\mathrm{(b)}}{=}\mathbb{E}_{x}\left[1_{\{\xi_1< T\}}\exp\left(-\int_0^{\xi_1}Y_sds+\xi_1\lambda(x)\right)\frac{\lambda(Y_{\xi_1})}{\lambda(x)}\zeta_x(Y_{\xi_1}+V_1,\xi_1)\right]\\
&\overset{\mathrm{(c)}}{=}\mathbb{E}_{x}^{\mathbb{Q}}\left[1_{\{T_1< T\}}\exp\left(-\int_0^{T_1}X_sds+\xi_1\lambda(x)\right)\frac{\lambda(X_{T_1}^{-})}{\lambda(x)}\zeta_x(X_{T_1},T_1)\right]\\
&\overset{\mathrm{(d)}}{=}\mathbb{E}_{x}\left[1_{\{T_1< T\}}\zeta_x(X_{T_1},T_1)\right]\\
&= \mathbb{E}_{x}\left[1_{\{T_1< T\}}\mathbb{E}_{X_{T_1}}[f(\bar{X}_{T-T_1})1_{\bar{A}_{n, T - T_1}}]\right]\\
&= \mathbb{E}_{x}\left[1_{\{T_1< T\}}1_{A_{n+1, T}}f(X_{T})\right] = \mathbb{E}_{x}\left[1_{A_{n+1, T}\cap A_{1,T}^c}f(X_{T})\right],
\end{aligned}
\end{equation*}
where $(a)$ comes from $\eqref{L_process}$ and $\eqref{L_process_theta}$, $(b)$ comes from the tower property, $(c)$ is because the dynamics of $X$ and $Y$ from time $(0, T_1)$ under measure $\mathbb{Q}$ is the same, and $(d)$ is the change of measure formula.

\section{Proof: Finite Variance}
From formula $\eqref{recursive}$, by denoting $X_0 = x$, $X_{i+1}^\pi = Y^{\pi,i+1}_{\xi_{i+1}} +V_{i+1}^\pi$, $T_i^\xi = \sum_{j=1}^i\xi_j$, and $N_T^\xi = \max\{n:T_n^\xi < T\}$. From equation above we can write $U(x, T)$ as

\begin{equation*}
\begin{aligned}	
U(x, T) &= \left(\prod_{i=1}^{N^{\mathbb{\xi}}_T}\Xi_1(X_{i-1}^\pi,\xi_i,Z^{\pi, i})\right)\Xi_2(X_{N_T^\xi}^\pi, T - T_{N_T^\xi}^\xi, Z^{\pi, N_T^\xi+1}).
\end{aligned}
\end{equation*}
Since the dynamic of our approximated Euler process directly controls the intensity, which affects the counting process $N_T^\xi$, if we directly condition on $N_T^\xi$, the law of the Euler process might become intractable. Therefore, we use thinning, a standard technique in simulation, to disentangle this dependency. Observe that $\xi_i$ is sampled from exponential distribution with intensity $\lambda(X_i^\pi)$, and that $\lambda_1 \leq \lambda(x) \leq \lambda_2$ for any $x \in \mathbb{R}^d$. We know that from thinning, there exists $\xi_i^-$ such that for every $i$, there exist strictly positive $k_i$ and $c_i$ such that $\xi_i = \sum_{j = k_i}^{k_i+c_i}\xi_j^-$, and the cumulative distribution funciton of $\xi_i^-$ is $1-e^{-\lambda_2 t}$. We denote $T_n^- = \sum_{i=1}^n \xi_i^-$ and $N_T^+ = \max\{n:T_n^- < T\}$. We consider bounding the expectation of $U(x, T)$ conditioned on $N_T^+$ so that 
\begin{equation*}
\begin{aligned}	
\mathbb{E}[U(x, T)^2] &= \sum_{m=0}^{+\infty} \mathbb{E}\left[\left.U(x, T)^2\right|N_T^+ = m\right]P(N_T^+ = m)
\end{aligned}
\end{equation*}
will be bounded. The advantage of conditioning on $N_T^+$ instead of $N_T^\xi$ is that we can disentangle the effect of the number of jumps $N_T^\xi$ on the approximation process $Y^\pi$, hence we can reuse the previous bound on the gaussian density of $Y^\pi$. Under the event $\{N_T^+ = m\}$, we have
$T_{m}^- < T < T_{m+1}^-$ almost surely.
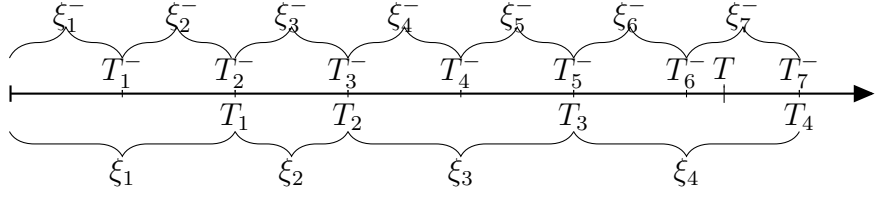
\begin{figure}[t]
\centering
    \begin{tikzpicture}[
      minimum size=0.5cm,
      inner sep=0pt,
      yscale=0.35
    ]

      \draw[thick, |->] (-1.5, -7) node[anchor=east] {} -- (10, -7)
          node[anchor=west] {};
      
      \draw[
        decorate, decoration={brace, amplitude=1em, aspect=32/64}, yshift=10pt
      ] (-1.5, -6) -- (0, -6);
      \node at (-0.70, -4) {$\xi_1^-$};
      
      \draw[
        decorate, decoration={brace, amplitude=1em, aspect=32/64}, yshift=10pt
      ] (0, -6) -- (1.45, -6);
      \node at (0.80, -4) {$\xi_2^-$};

      \draw (0, -7cm + 4pt) -- (0, -7cm - 4pt) node[anchor=north]
          {};
      \draw (0, -7cm + 2pt) -- (0, -6.8cm - 2pt) node[anchor=south]
          {\(T_1^-\)};

      \draw (1.5, -7cm + 4pt) -- (1.5, -7cm - 4pt) node[anchor=north]
          {\(T_{1}\)};
          
      \draw (1.5, -7cm + 2pt) -- (1.5, -6.8cm - 2pt) node[anchor=south]
          {\(T_2^-\)};
      
      \draw[
        decorate, decoration={brace, amplitude=0.8em, aspect=6/12, mirror},
        yshift=-4pt
      ] (-1.5, -8.3) -- (1.5, -8.3);
      \node at (0, -10) {$\xi_1$};
      
      \draw (3.0, -7cm + 4pt) -- (3.0, -7cm - 4pt) node[anchor=north]
          {\(T_{2}\)};
          
      \draw (3.0, -7cm + 2pt) -- (3.0, -6.8cm - 2pt) node[anchor=south]
          {\(T_3^-\)};
          
      \draw[
        decorate, decoration={brace, amplitude=1em, aspect=32/64}, yshift=10pt
      ] (1.5, -6) -- (3.0, -6);
      \node at (2.25, -4) {$\xi_3^-$};
      
      \draw[
        decorate, decoration={brace, amplitude=0.8em, aspect=6/12, mirror},
        yshift=-4pt
      ] (1.5, -8.3) -- (3, -8.3);
      \node at (2.25, -10) {$\xi_2$};

      \draw (4.5, -7cm + 4pt) -- (4.5, -7cm - 4pt) node[anchor=north]
          {};
          
      \draw (4.5, -7cm + 2pt) -- (4.5, -6.8cm - 2pt) node[anchor=south]
          {\(T_4^-\)};
          
      \draw[
        decorate, decoration={brace, amplitude=1em, aspect=32/64}, yshift=10pt
      ] (3.0, -6) -- (4.5, -6);
      \node at (3.75, -4) {$\xi_4^-$};

      \draw (6.0, -7cm + 4pt) -- (6.0, -7cm - 4pt) node[anchor=north]
          {\(T_{3}\)};
          
      \draw (6.0, -7cm + 2pt) -- (6.0, -6.8cm - 2pt) node[anchor=south]
          {\(T_5^-\)};
          
      \draw[
        decorate, decoration={brace, amplitude=1em, aspect=32/64}, yshift=10pt
      ] (4.5, -6) -- (6.0, -6);
      \node at (5.25, -4) {$\xi_5^-$};
      
      \draw[
        decorate, decoration={brace, amplitude=0.8em, aspect=6/12, mirror},
        yshift=-4pt
      ] (3, -8.3) -- (6, -8.3);
      \node at (4.5, -10) {$\xi_3$};

      \draw (7.5, -7cm + 4pt) -- (7.5, -7cm - 4pt) node[anchor=north]
          {};
          
      \draw (7.5, -7cm + 2pt) -- (7.5, -6.8cm - 2pt) node[anchor=south]
          {\(T_6^-\)};
          
      \draw[
        decorate, decoration={brace, amplitude=1em, aspect=32/64}, yshift=10pt
      ] (6.0, -6) -- (7.5, -6);
      \node at (6.75, -4) {$\xi_6^-$};

     \draw (8, -7cm + 10pt) -- (8, -7cm - 10pt) node[anchor=north]
  {};
          
      \draw (8, -7cm + 8pt) -- (8, -6.7cm - 2pt) node[anchor=south]
          {\(T\)};

      \draw (9, -7cm + 4pt) -- (9, -7cm - 4pt) node[anchor=north]
          {\(T_4\)};
          
      \draw (9, -7cm + 2pt) -- (9, -6.8cm - 2pt) node[anchor=south]
          {\(T_7^-\)};

    \draw[
        decorate, decoration={brace, amplitude=1em, aspect=32/64}, yshift=10pt
      ] (7.5, -6) -- (9, -6);
      \node at (8.25, -4) {$\xi_7^-$};
      
      \draw[
        decorate, decoration={brace, amplitude=0.8em, aspect=6/12, mirror},
        yshift=-4pt
      ] (6, -8.3) -- (9, -8.3);
      \node at (7.5, -10) {$\xi_4$};

    %   \draw (1.5, -7cm + 4pt) -- (2, -7cm - 4pt);
    %   \draw (5, -7cm + 4pt) -- (5, -7cm - 4pt) node[anchor=north]
    %       {\(\tau_{2}-1\)};
          
    %   \draw[
    %     decorate, decoration={brace, amplitude=1em, aspect=32/64}, yshift=10pt
    %   ] (5.5, -6) -- (7, -6);
    %   \node at (6.25, -4) {$\mathcal{B}_2$};
    
    %   \draw (5.5, -7cm + 4pt) -- (5.5, -7cm - 4pt) node[anchor=south]
    %       {\(\tau_{2}\)};
          
    %   \draw[
    %     decorate, decoration={brace, amplitude=1em, aspect=32/64}, yshift=10pt
    %   ] (7.5, -6) -- (10.5, -6);
    %   \node at (8.75, -4) {$\mathcal{A}_2$};
          
    %   \draw (7, -7cm + 4pt) -- (7, -7cm - 4pt) node[anchor=north]
    %       {\(t_{2}\)};
    %   \draw(7.5, -7cm + 4pt) -- (7.5, -7cm - 4pt);
    %   \draw (10.5, -7cm + 4pt) -- (10.5, -7cm - 4pt) node[anchor=north]
    %       {\(\tau_{3}-1\)};
    %   \draw (11, -7cm + 4pt) -- (11, -7cm - 4pt) node[anchor=south]
    %       {\(\tau_{3}\)};

    \end{tikzpicture}
    \caption{Example from thining. Notice that in this particular example, we have that $C(1) = C(2) = 1$, $C(3) = 2$, $C(4) = C(5) = C(6) = C(7) = 3$, $N_T = 3$ and $N_T^+ = 6$.}
\label{fig_thinning}
\end{figure}
Then, we define $C(l) = \min\{\min\{k: T_k \geq T_l^-\}, N_T^\xi\}$ (see Figure \ref{fig_thinning} for an example). For a shorthand denote $\mathbb{E}_{m,+}[\cdot]$ to be $\mathbb{E}[\cdot|N_T^+ = m]$, and $Y_t^{\pi,i,(j)}$ to denote the $i$-th segment (corresponding to the Euler approximation in the period $(T_i^{\xi}, T_{i+1}^{\xi})$) of the Euler approximation for $X$ in dimension $j$, evaluated at $t$. We use the following lemma for our proof.
\begin{lemma} \label{lemma_recursive_variance}
For $T_{m}^- < T < T_{m+1}^-$ and $f$ with exponential growth, we have 
\begin{equation}\label{ineq_variance_final_term}
\begin{aligned}	
\mathbb{E}_{m,+}\left[U(x, T)^2\right] &\leq M_T\exp(2c_1vd-\sigma_A^2(T-T_{N_T^\xi}^\xi))\\
&\hspace{10mm}\cdot\mathbb{E}_{m,+}\left[\prod_{i=1}^{N_T^\xi}\Xi_1(X_{i-1}^\pi,\xi_i,Z^{\pi, i})^2\prod_{j=1}^d \cosh\left(2c_1Y_{\xi_{N_T^\xi}}^{\pi, N_T^\xi,(j)}\right)\right],
\end{aligned}
\end{equation}
Moreover, for $l \geq 1$ we have
\begin{equation} \label{ineq_variance_previous_term}
\begin{aligned}	
&\hspace{5mm}\mathbb{E}_{l,+}\left[\prod_{i=1}^{C(l)}\Xi_1(X_{i-1}^\pi,\xi_i,Z^{\pi, i})^2\prod_{j=1}^d \cosh\left(2c_1Y_{\xi_{C(l)}}^{\pi, C(l),(j)}\right)\right]\\
&\leq 
M_T\exp(2c_1vd)\left(\frac{\lambda_2}{\lambda_1}\right)^2\mathbb{E}_{l,+}\left[\prod_{i=1}^{C(l-1)}\Xi_1(X_{i-1}^\pi,\xi_i,Z^{\pi, i})^2\prod_{j=1}^d \cosh(2c_1Y_{\xi_{C(l-1)}}^{\pi, C(l-1),(j)})\right]
\end{aligned}
\end{equation}
\begin{proof}
Since $\{T_i^-\}_{i=1}^{m+1}$ are sampled before $\{Z^{\pi, i}\}_{i=1}^{N_T^\xi}$ are sampled, conditioning on $\{N_T^+ = m\}$ does not affect the law of the Brownian Motion. Under the event $\{N_T^+ = m\}$, we have $\mathcal{F}_{T_{N_T^\xi}} \subseteq \mathcal{F}_{T_m^-}$, hence
\begin{equation*}
\begin{aligned}	
\mathbb{E}_{m,+}\left[U(x, T)^2\right] = \mathbb{E}_{m,+}\left[\prod_{i=1}^{N_T^\xi}\Xi_1(X_{i-1}^\pi,\xi_i,Z^{\pi, i})^2\mathbb{E}_{m,+}\left[\left.\Xi_2(X_{N_T^\xi}^\pi, T - T_{N_T^\xi}^\xi, Z^{\pi, N_T^\xi+1})^2\right|\mathcal{F}_{T_{N_T^\xi}}\right]\right].
\end{aligned}
\end{equation*}
From Condition \ref{cdn_var}, strong Markov property, and the definition of $\Xi_2$ we have 

\begin{equation}\label{eqn_bound_jump}
\begin{aligned}	
&\hspace{6mm}\mathbb{E}_{m,+}\left[\left.\Xi_2(X_{N_T^\xi}^\pi, T - T_{N_T^\xi}^\xi, Z^{\pi, N_T^\xi+1})^2\right|\mathcal{F}_{T_{N_T^\xi}}\right]\\ 
&\leq M_T\exp(\sigma_A^2(T-T_{N_T^\xi}^\xi))\mathbb{E}_{m,+}\left[\left.\prod_{j=1}^d \cosh\left(2c_1\left(Y^{\pi,N_T^\xi}_{\xi_{N_T^\xi}} +V_{N_T^\xi}^{\pi}\right)^{(j)}\right)\right|\mathcal{F}_{T_{N_T^\xi}}\right]\\
&\leq M_T\exp(2c_1vd+\sigma_A^2(T-T_{N_T^\xi}^\xi))\mathbb{E}_{m,+}\left[\left.\prod_{j=1}^d \cosh\left(2c_1Y^{\pi,N_T^{\xi},(j)}_{\xi_{N_T^{\xi}}}\right)\right|\mathcal{F}_{T_{N_T^{\xi}}}\right],
\end{aligned}
\end{equation}
where the last inequality comes from the Assumption \ref{jasm} that $|V_i^{\pi}| = |h(Y_{\xi_i}^{\pi, i}, R_i)|  \leq v$, and we are done with proving \eqref{ineq_variance_final_term}. For proving the second inequality \eqref{ineq_variance_previous_term}, we find that by definition of $C(l)$, for any $l > 0$, we will always have either $T_{C(l) - 1} = T_{l-1}^-$ or $T_{C(l) - 1} < T_{l-1}^-$. Hence we have
\begin{equation}\label{eqn_var_recursion_question}
\begin{aligned}	
&\hspace{6mm}\mathbb{E}_{l,+}\left[\prod_{i=1}^{C(l)}\Xi_1(X_{i-1},\xi_i,Z^{\pi, i})^2\prod_{j=1}^d \cosh\left(2c_1Y_{\xi_{C(l)}}^{\pi, C(l),(j)}\right)\right]\\ 
&= \mathbb{E}_{l,+}\left[1_{\{T_{C(l) - 1} = T_{l-1}^-\}}\prod_{i=1}^{C(l)}\Xi_1(X_{i-1},\xi_i,Z^{\pi, i})^2\prod_{j=1}^d \cosh\left(2c_1Y_{\xi_{C(l)}}^{\pi, C(l),(j)}\right)\right]\\
&\hspace{10mm}+ \mathbb{E}_{l,+}\left[1_{\{T_{C(l) - 1} < T_{l-1}^-\}}\prod_{i=1}^{C(l)}\Xi_1(X_{i-1},\xi_i,Z^{\pi, i})^2\prod_{j=1}^d \cosh\left(2c_1Y_{\xi_{C(l)}}^{\pi, C(l),(j)}\right)\right].
\end{aligned}
\end{equation}
Again, from Condition \ref{cdn_var}, the definition of $\Xi_1$ and the same approach in \eqref{eqn_bound_jump} we have
\begin{equation*}
    \begin{aligned}
    &\hspace{5mm}\mathbb{E}_{l,+}\left[\left.\Xi_1(X_{C(l)-1},\xi_{C(l)},Z^{\pi, C(l)})^2\prod_{j=1}^d \cosh\left(2c_1Y_{\xi_{C(l)}}^{\pi, C(l),(j)}\right)\right|\mathcal{F}_{T_{C(l)-1}}\right]\\
    &\leq M_T\exp(2c_1vd-\sigma_A^2\xi_{C(l)})\left(\frac{\lambda_2}{\lambda_1}\right)^2\prod_{j=1}^d \cosh\left(2c_1Y_{\xi_{C(l)-1}}^{\pi, C(l)-1,(j)}\right).
    \end{aligned}
\end{equation*}
Next, from tower property and strong Markov property, by conditioning on $\mathcal{F}_{T_{C(l)-1}}$ we know that
\begin{equation*}
\begin{aligned}	
&\hspace{5mm}\mathbb{E}_{l,+}\left[1_{\{T_{C(l) - 1} = T_{l-1}^-\}}\prod_{i=1}^{C(l)}\Xi_1(X_{i-1},\xi_i,Z^{\pi, i})^2\prod_{j=1}^d \cosh\left(2c_1Y_{\xi_{C(l)}}^{\pi, C(l),(j)}\right)\right]\\
&\leq M_T\exp(2c_1vd-\sigma_A^2\xi_{C(l)})\left(\frac{\lambda_2}{\lambda_1}\right)^2\\
&\hspace{10mm}\cdot\mathbb{E}_{l,+}\left[1_{\{T_{C(l) - 1} = T_{l-1}^-\}}\prod_{i=1}^{C(l)-1}\Xi_1(X_{i-1},\xi_i,Z^{\pi, i})^2\prod_{j=1}^d \cosh\left(2c_1Y_{\xi_{C(l)-1}}^{\pi, C(l)-1,(j)}\right)\right].
\end{aligned}
\end{equation*}
Also observe that when $T_{C(l) - 1} = T_{l-1}^-$, we have $C(l)-1 = C(l-1)$, hence 
\begin{equation}\label{eqn_var_recursion_result_1}
\begin{aligned}	
&\hspace{5mm}\mathbb{E}_{l,+}\left[1_{\{T_{C(l) - 1} = T_{l-1}^-\}}\prod_{i=1}^{C(l)}\Xi_1(X_{i-1},\xi_i,Z^{\pi, i})^2\prod_{j=1}^d \cosh\left(2c_1Y_{\xi_{C(l)}}^{\pi, C(l),(j)}\right)\right]\\
&\leq M_T\exp(2c_1vd-\sigma_A^2\xi_{C(l)})\left(\frac{\lambda_2}{\lambda_1}\right)^2\\
&\hspace{10mm}\cdot\mathbb{E}_{l,+}\left[1_{\{T_{C(l-1)} = T_{l-1}^-\}}\prod_{i=1}^{C(l-1)}\Xi_1(X_{i-1},\xi_i,Z^{\pi, i})^2\prod_{j=1}^d \cosh\left(2c_1Y_{\xi_{C(l-1)}}^{\pi, C(l-1),(j)}\right)\right].
\end{aligned}
\end{equation}
For the second term in \eqref{eqn_var_recursion_question}, if $T_{C(l) - 1} < T_{l-1}^-$, we have $C(l-1) = C(l)$. It implies
\begin{equation}\label{eqn_var_recursion_result_2}
\begin{aligned}	
&\hspace{6mm}\mathbb{E}_{l,+}\left[1_{\{T_{C(l) - 1} < T_{l-1}^-\}}\prod_{i=1}^{C(l)}\Xi_1(X_{i-1},\xi_i,Z^{\pi, i})^2\prod_{j=1}^d \cosh\left(2c_1Y_{\xi_{C(l)}}^{\pi, C(l),(j)}\right)\right]\\
&\leq  M_T\exp(2c_1vd-\sigma_A^2\xi_{C(l)})\left(\frac{\lambda_2}{\lambda_1}\right)^2\\
&\hspace{10mm}\mathbb{E}_{l,+}\left[1_{\{T_{C(l) - 1} < T_{l-1}^-\}}\prod_{i=1}^{C(l-1)}\Xi_1(X_{i-1},\xi_i,Z^{\pi, i})^2\prod_{j=1}^d \cosh\left(2c_1Y_{\xi_{C(l-1)}}^{\pi, C(l-1),(j)}\right)\right],
\end{aligned}
\end{equation}
where in the last equation WLOG we assume $ M_T\exp(2c_1vd-\sigma_A^2\xi_{C(m)})\left(\frac{\lambda_2}{\lambda_1}\right)^2 > 1$. Combining \eqref{eqn_var_recursion_question}, \eqref{eqn_var_recursion_result_1}, and \eqref{eqn_var_recursion_result_2} will prove the second equation stated in the lemma, thereby finishing the proof of lemma \ref{lemma_recursive_variance}.
\end{proof}
\end{lemma}

Using Lemma \ref{lemma_recursive_variance} with the observation that $C(m) = N_T^{\xi}$ and $C(1) = 1$, conditioned on $\{N_T^+ = m\}$ we have 

\begin{equation*}
\begin{aligned}	
\mathbb{E}_{m,+}\left[U(x, T)^2\right] \leq \exp(-\sigma_A^2T)\left( M_T\exp(2c_1vd)\left(\frac{\lambda_2}{\lambda_1}\right)^2\right)^{m+1}\prod_{j=1}^d \cosh\left(2c_1x^{(j)}\right).
\end{aligned}
\end{equation*}
Lastly,
\begin{equation*}
\begin{aligned}	
\mathbb{E}[U(x, T)^2] &= \sum_{m=0}^{+\infty} \mathbb{E}\left[\left.U(x, T)^2\right|N_T^+ = m\right]P(N_T^+ = m)\\
&\leq \exp(-\sigma_A^2T)\prod_{j=1}^d \cosh\left(2c_1x^{(j)}\right)\\
&\hspace{10mm}\cdot\sum_{m=0}^{+\infty}\left(M_T\exp(2c_1vd)\left(\frac{\lambda_2}{\lambda_1}\right)^{2}\right)^{m+1}\frac{(\lambda_2T)^me^{-\lambda_2T}}{m!} < +\infty.
\end{aligned}
\end{equation*}

\section{Other Lemmas}
For notational simplicity, in this section it is convenient to denote $X = (Y, Z) \in \mathbb{R}^{d+1}$ instead of $Z = (Y, \bar{A}) \in \mathbb{R}^{d+1}$. The reason is that sometimes we do not have to view the process as a jump-diffusion process, and we can save the notation for $\bar{A}$, which might confuse reader with the volatility matrix $a(y) = \sigma(y)^\top\sigma(y)$. Notice that we use the lower case $x = (y, z)$ to denote the the realized sample. Let $\varphi_c$ denote the multivariate Gaussian density
with a zero mean and variance (matrix) $c$.
Denote by $q$ the transition kernel of the Euler
process $X^\pi = (Y^\pi,Z^\pi)$ defined in on some interval
$[\tau_k,\tau_{k+1})$, given
$(X^\pi_{\tau_k},\tau_k,\tau_{k+1})$.
The law of $(Y^\pi,Z^\pi)$ given 
$X^\pi_{\tau_k} = (y_1,z_1) \in \rn^{d+1}$ is Gaussian
with covariances $a(y_1) = (\sigma\sigma^\top)(y_1) 
\in \rn^{d\times d}$
and $\sigma_A^2$ for $Y^\pi$ and $Z^\pi$ respectively
with means $y_1 + \mu(y_1)$ and $z_1 + \lambda(z_1)$.
For initial and final points 
$x_1 = (y_1, z_1)$ and 
$x_2 = (y_2, z_2) \in \rn^{d+1}$, the 
density $q_t(x_1, x_2)$ decomposes as
\begin{align} \label{Xplaw}
  q_t(x_1, x_2) = \varphi_{t a(y_1)} (y_2-y_1-\mu(y_1) t)
\varphi_{t \sigma_A^2} (z_2 - z_1 - \lambda(y_1)t)
\end{align}
as the Brownian motions $W$ and $B$ 
driving $Y^\pi$ and $Z^\pi$ are independent. Moreover, with a slight abuse of notation, sometimes we use $q_t$ to denote the transition kernel of the $y$ or $z$ component separately, i.e.
$$  q_t(y_1, y_2) = \varphi_{t a(y_1)} (y_2-y_1-\mu(y_1) t)$$
$$  q_t(z_1, z_2) = \varphi_{t \sigma_A^2} (z_2 - z_1 - \lambda(y_1)t).$$

Then we begin to introduce some lemmas.

\begin{lemma} \label{lemma_bound_thetaq}
Under Assumption \ref{dasm} and \ref{jasm}, for $x_i = (y_i, z_i)$ and $i \in \{1,2\}$ there exists $C_T$ such that for all $t < T$ we have
\begin{align}\label{ine_theta}
|\theta_t(x_1, x_2)^pq_t(x_1,x_2)| &\leq \cfrac{C_T}{t^{p/2}}\varphi_{4a_2t}(y_2-y_1)\varphi_{2\sigma_A^2t}(z_2-z_1-\lambda(y_1)t).
\end{align}
\end{lemma}

\begin{proof}[Proof of Lemma \ref{lemma_bound_thetaq}]
Firstly, by definition of $\theta$ in $\eqref{eqn:def_theta}$, we have
\begin{equation}\label{pf_lemma_2_first_ineq}
\begin{aligned}
&\hspace{6mm}|\theta_t(x_1, x_2)q_t(x_1, x_2)^{1/p}|\\
&\leq |\vartheta_t(y_1, y_2)(q_t(y_1, y_2)q_t(z_1, z_2))^{1/p}| + \left|(\lambda(y_2) - \lambda(y_1))q_t(y_1, y_2)^{1/p}\cfrac{z_2 - z_1 - \lambda(y_1)t}{\sigma_A^2t}q_t(z_1, z_2)^{1/p}\right|
\end{aligned}
\end{equation}
The plan is to bound the two terms in $\eqref{pf_lemma_2_first_ineq}$. For the first one, from Corollary 4.2 in \citet{anderssonhiga2017}, we know that there exist $C_T' > 0$ such that
\begin{equation}\label{pf_lemma_2_term_1_1}
\begin{aligned}
|\vartheta_t(y_1, y_2)q_t(y_1, y_2)^{1/p}| &\leq \left(\cfrac{C_T'}{t^{p/2}}\varphi_{4a_2t}(y_2-y_1)\right)^{1/p}.
\end{aligned}
\end{equation}

We can also derive that 

\begin{equation}\label{pf_lemma_2_term_1_2}
\begin{aligned}
q_t(z_1, z_2) &= \cfrac{1}{\sqrt{2\pi\sigma_A^2 t}}\text{exp}\left(-\cfrac{(z_2 - z_1 - \lambda(y_1)t)^2}{2\sigma_A^2t}\right) \leq \cfrac{\sqrt{2}}{\sqrt{2\pi2\sigma_A^2 t}}\text{exp}\left(-\cfrac{(z_2 - z_1 - \lambda(y_1)t)^2}{2\cdot2\sigma_A^2t}\right)\\
& \leq \sqrt{2} \varphi_{2\sigma_A^2 t}(z_2 - z_1 - \lambda(y_1)t).
\end{aligned}
\end{equation}

Combining $\eqref{pf_lemma_2_term_1_1}$ and $\eqref{pf_lemma_2_term_1_2}$ we have

\begin{equation}\label{pf_lemma_2_term_1_final}
\begin{aligned}
|\vartheta_t(y_1, y_2)(q_t(y_1, y_2)q_t(z_1, z_2))^{1/p}| \leq \cfrac{(\sqrt{2}C_T')^{1/p}}{t^{1/2}}\varphi_{4a_2t}(y_2-y_1)^{1/p}\varphi_{2\sigma_A^2 t}(z_2 - z_1 - \lambda(y_1))^{1/p}
\end{aligned}
\end{equation}

For the second term in $\eqref{pf_lemma_2_first_ineq}$, we first denote that since $\lambda(x)$ is bounded, there exists $l_1 > 0$ such that $||\lambda(y_1) - \lambda(y_2)||_2 \leq l_1||y_1 - y_2||_2$. Moreover, from Lemma A.1 in \citet{anderssonhiga2017} we know there exist $C_T^{(2)} = 2^{d/2}e^{||\mu||_0Ta_2^{-1}/4}$ such that

\begin{equation}\label{pf_lemma_2_term_2_fisrt}
\begin{aligned}
& \hspace{6mm}|(\lambda(y_1) - \lambda(y_2))q_t(y_1, y_2)^{1/p}|\\
&\leq l_1||y_2 - y_1||_2q_t(y_1, y_2)^{1/p}\leq l_1C_T^{(2)}||y_2 - y_1||_2\varphi_{2ta(y_1)}(y_2 - y_1)^{1/p}\\
&\leq l_1C_T^{(2)}(2a_2/a_1)^{d/p}(4Ta_2p)\varphi_{4ta_2}(y_2 - y_1 )^{1/p}\\
\end{aligned}
\end{equation}

We apply the same procedure to the second inequality,
\begin{equation}\label{pf_lemma_2_term_2_second}
\begin{aligned}
& \hspace{6mm}\left|\cfrac{z_2 - z_1 - \lambda(y_1)t}{\sigma_A^2t}q_t(z_1, z_2)^{1/p}\right| = \cfrac{\left|z_2 - z_1 - \lambda(y_1)t\right|}{\sigma_A^2t}\cfrac{1}{(2\pi\sigma_A^2 t)^{1/2p}}\text{exp}\left(-\cfrac{(z_2 - z_1 - \lambda(y_1)t)^2}{2\sigma_A^2tp}\right)\\
&\leq \cfrac{1}{\sigma_A^2 t}\varphi_{2\sigma_A^2t}(z_2-z_1-\lambda(y_1)t)^{1/p}\left(2^{1/2p}|z_2 - z_1 -\lambda(y_1)t|\text{exp}\left(-\cfrac{(z_2 - z_1 - \lambda(y_1)t)^2}{4\sigma_A^2tp}\right)\right)\\
&\leq \cfrac{1}{\sigma_A^2 t}\varphi_{2\sigma_A^2t}(z_2-z_1-\lambda(y_1)t)^{1/p}(4\cdot2^{1/2p}\sigma_A^2tp)\\
&\leq \cfrac{T^{1/2}4p2^{1/2p}}{t^{1/2}}\varphi_{2\sigma_A^2t}(z_2-z_1-\lambda(y_1)t)^{1/p},
\end{aligned}
\end{equation}
where the second inequality come from the fact $|x|e^{-x^2}$ is bounded. Therefore, if we combine $\eqref{pf_lemma_2_term_2_fisrt}$ and $\eqref{pf_lemma_2_term_2_second}$ we can have
\begin{equation}\label{pf_lemma_2_term_2_final}
\begin{aligned}
&\hspace{5mm}\left|(\lambda(y_2) - \lambda(y_1))q_t(y_1, y_2)^{1/p}\cfrac{z_2 - z_1 - \lambda(y_1)t}{\sigma_A^2t}q_t(z_1, z_2)^{1/p}\right|\\
&\leq \cfrac{C_T^{(3)}}{t^{1/2}}\varphi_{4a_2t}(y_2 - y_1)^{1/p}\varphi_{2\sigma_A^2t}(z_2 - z_1 - \lambda(y_1)t)^{1/p},
\end{aligned}
\end{equation}
where
\begin{equation*}
\begin{aligned}
C_T^{(3)} = T^{1/2}4p2^{1/2p}l_1C_T^{(2)}(2a_2/a_1)^{d/p}(4Ta_2p).
\end{aligned}
\end{equation*}
Lastly, if we combine $\eqref{pf_lemma_2_first_ineq}$, $\eqref{pf_lemma_2_term_1_final}$ and $\eqref{pf_lemma_2_term_2_final}$ we can have 
\begin{equation*}
\begin{aligned}
|\theta_t(x_1, x_2)q_t(x_1, x_2)^{1/p}| \leq 
\cfrac{C_T^{(3)} + (\sqrt{2}C_T')^{1/p}}{t^{1/2}}\varphi_{4a_2t}(y_2-y_1)^{1/p}\varphi_{2\sigma_A^2 t}(z_2 - z_1 - \lambda(y_1)t)^{1/p},
\end{aligned}
\end{equation*}
thereby finishing the proof.
\end{proof}

\begin{lemma}\label{absolute inequality}
For $a, b > 0$ and $f$ being integrable, we have
\begin{align}
    \int_{+\infty}^{+\infty} e^{|a|x}(e^{bx} + e^{-bx})|f(x)|dx \leq\int_{+\infty}^{+\infty} 2(e^{(a+b)x} + e^{-(a+b)x})|f(x)|dx.
\end{align}
\begin{proof}
\begin{equation*}
    \begin{aligned}
    &\hspace{5mm}\int_{+\infty}^{+\infty} e^{|a|x}(e^{bx} + e^{-bx})|f(x)|dx\\
    &= \int_0^{+\infty}(e^{(a+b)x}+ e^{-(b-a)x})|f(x)|dx+\int_{+\infty}^0(e^{(b-a)x}+ e^{-(b+a)x})|f(x)|dx dx\\
    &\leq\int_{+\infty}^{+\infty}(e^{(a+b)x}+ e^{-(b-a)x})|f(x)|dx+\int_{+\infty}^{+\infty}(e^{(b-a)x}+ e^{-(b+a)x})|f(x)| dx\\
    &\leq2\int_{+\infty}^{+\infty}(e^{(a+b)x}+ e^{-(a+b)x})|f(x)|dx
    \end{aligned}
\end{equation*}
The last inequality holds because the function $e^{x} + e^{-x}$ is odd and monotone on $(0, \infty)$.
\end{proof}
\end{lemma}

\begin{lemma}\label{integration}
Let $\varphi_b(x)$ be the one dimensional gaussian density 
\begin{equation*}
    \begin{aligned}
    \varphi_b(x) = \cfrac{1}{\sqrt{2\pi b}}\,e^{-x^2/2b}.
    \end{aligned}
\end{equation*}
We have
\begin{equation*}
    \begin{aligned}
    \int_{+\infty}^{+\infty}e^{ax}\varphi_b(x - y)dx &= e^{a^2b/2}e^{ay}\\
    \int_{+\infty}^{+\infty}2cosh(ax)\varphi_b(x - y)dx &= e^{a^2b/2}2cosh(ay)
    \end{aligned}
\end{equation*}
\begin{proof}
\begin{equation*}
    \begin{aligned}
    \int_{+\infty}^{+\infty}e^{ax}\varphi_b(x - y)dx  &= \int_{+\infty}^{+\infty}\cfrac{1}{\sqrt{2\pi b}}\,e^{-((x-y)^2-2abx)/2b}dx\\
    &= \int_{+\infty}^{+\infty}\cfrac{1}{\sqrt{2\pi b}}\,e^{-((x-(y+ab))^2-2yab-a^2b^2)/2b}dx\\
    &= \int_{+\infty}^{+\infty}\cfrac{1}{\sqrt{2\pi b}}\,e^{-(x-(y+ab))^2/2b}dxe^{ay}e^{a^2b/2} = e^{ay}e^{a^2b/2}\\
    \end{aligned}
\end{equation*}
For the same reason we have
\begin{equation*}
    \begin{aligned}
    \int_{+\infty}^{+\infty}e^{-ax}\varphi_b(x - y)dx &= e^{-ay}e^{a^2b/2}\\
    \end{aligned}
\end{equation*}
thereby finishing the proof.
\end{proof}
\end{lemma}

\subsection{Proof of  Lemma \ref{lemma_para_var}}
We denote $dx_n = dx_{s_n} \cdots dx_{s_1}$, where $x = (y, z)$, and denote $\xi_{i+1} = s_{i+1} - s_{i}$ for $i \leq n+1$. Also notice that we denote $T = s_{n+1}$ for notational convenience. We have
\begin{equation*}
\begin{aligned}
&\hspace{5mm}\mathbb{E}\left[\left|\frac{e^{-Z^{\pi}_T+c_1||Y_T^\pi||_1+c_2}}{p_n(s_1, \cdots , s_n)}\prod_{i=1}^{n}\theta_{s_i - s_{i-1}}(X_{s_{i-1}}^{\pi}, X_{s_i}^{\pi})\right|^p\right]\\
&=\int \cdots \int \frac{e^{-pz_T+c_1p||y_T||_1+c_2p}}{p_n(s_1, \cdots , s_n)^p}q_{T-s_n}(x_{s_n}, x_T)\\
&\hspace{10mm}\cdot\left(\prod_{i=1}^{n}|\theta_{s_i-s_{i-1}}(x_{s_{i-1}}, x_{s_i})|^pq_{s_i-s_{i-1}}(x_{s_{i-1}}, x_{s_i})\right)dx_Tdx_{s_n}\cdots dx_{s_1}.\\
\end{aligned}
\end{equation*}
We first integrate $x_T = (y_T, z_T)$, the corresponding component in the inner integral could be decomposed as
\begin{equation}\label{lem_3_0_1_decompose_z}
\begin{aligned}
&\hspace{5mm}\int e^{-pz_T+c_1p||y_T||_1+c_2p}q_{T-s_n}(x_{s_n}, x_T)dx_T\\ &= \int\int e^{-pz_T}q_{T-s_n}(z_{s_n}, z_T)dz_T\cdot e^{c_1p||y_T||_1+c_2p}q_{T-s_n}(y_{s_n}, y_T)dy_T\\
\end{aligned}
\end{equation}
Since $q_{T-s_n}(z_{s_n}, z_T)$ is the one dimensional gaussian density with mean $z_T - z_{s_n} - \lambda(y_{s_n})\xi_{n+1}$ and variance $\sigma_A^2\xi_{n+1}$, we have
\begin{equation}\label{notation_begin}
\begin{aligned}
\int e^{-p(z_T - z_{t_{n}} - \lambda(y_{s_n})\xi_{n+1})}q_{\xi_{n+1}}(z_{s_n}, z_T)dz_T = e^{\sigma_A^2p^2\xi_{n+1}/2}.
\end{aligned}
\end{equation}
For the cleanness of notation, we denote $M_p := e^{\sigma_A^2p^2T}$ and
\begin{equation}\label{notation_end}
\begin{aligned}
M(p,i) := e^{\sigma_A^2p^2\xi_{i}/2} \leq M_p.
\end{aligned}
\end{equation}
Then $\eqref{lem_3_0_1_decompose_z}$ becomes
\begin{equation}\label{lem_3_0_2_decompose_y}
\begin{aligned}
&\hspace{5mm}\int e^{-pz_T+c_1p||y_T||_1+c_2p}q_{\xi_{n+1}}(x_{s_n}, x_T)dx_T\\
&= M(p,n)e^{-pz_{s_n} - p\lambda(y_{s_n})\xi_{n+1}}\int e^{c_1p||y_T||_1+c_2p}q_{\xi_{n+1}}(y_{s_n}, y_T)dy_T\\
&\leq M_pe^{-pz_{s_n} - p\lambda(y_{s_n})\xi_{n+1}}\int e^{c_1p||y_T||_1+c_2p}q_{\xi_{n+1}}(y_{s_n}, y_T)dy_T.\\
\end{aligned}
\end{equation}
From Lemma A.1. in \citet{anderssonhiga2017} we know that there exist $C_T' = 2^{d/2}e^{\frac{1}{4}||\mu||_0Ta_2^{-1}}$ such that 
\begin{equation*}
\begin{aligned}
\int e^{c_1p||y_T||_1+c_2p}q_{\xi_{n+1}}(y_{s_n}, y_T)dy_T \leq C_T'\int e^{c_1p||y_T||_1+c_2p}\varphi_{2a_2\xi_{n+1}}(y_T - y_{s_n})dy_T.
\end{aligned}
\end{equation*}
Since now we are dealing with the multigaussian density of covariance matrix $2a_2\xi_{n+1}I$ we can integrate each component one by one. If we denote $y^{(i)}$ to be the $i$-th component, and with an abuse of notation we denote $\varphi_a$ to be the one dimensional gaussian density with variance $a$, we have 
\begin{equation}\label{lem_3_0_3_decompose_comp}
\begin{aligned}
\int e^{c_1p||y_T||_1+c_2p}q_{\xi_{n+1}}(y_{s_n}, y_T)dy_T &\leq e^{c_2p}C_T'\prod_{i=1}^d\left(\int e^{c_1p|y_T^{(i)}|}\varphi_{2a_2\xi_{n+1}}(y_T^{(i)} - y_{s_n}^{(i)})dy_T^{(i)}\right)\\
&\leq e^{c_2p}C_T'\prod_{i=1}^d\left(\int \left(e^{c_1py_T^{(i)}}+e^{-c_1py_T^{(i)}}\right)\varphi_{2a_2\xi_{n+1}}(y_T^{(i)} - y_{s_n}^{(i)})dy_T^{(i)}\right)\\
&= e^{c_2p}C_T'\prod_{i=1}^d\left(\int 2\cosh\left(c_1py_T^{(i)}\right)\varphi_{2a_2\xi_{n+1}}(y_T^{(i)} - y_{s_n}^{(i)})dy_T^{(i)}\right).\\
\end{aligned}
\end{equation}
By using Lemma \ref{integration} we have that 

\begin{equation}\label{lem_3_0_4_bound}
\begin{aligned}
\prod_{i=1}^d\left(\int \cosh\left(c_1py_T^{(i)}\right)\varphi_{2a_2\xi_{n+1}}(y_T^{(i)} - y_{s_n}^{(i)})dy_T^{(i)}\right) \leq e^{c_1^2p^2a_2d\xi_{n+1}}\prod_{i=1}^d \cosh\left(c_1py_{s_n}^{(i)}\right)
\end{aligned}
\end{equation}
By combining $\eqref{lem_3_0_1_decompose_z}$, $\eqref{lem_3_0_2_decompose_y}$, $\eqref{lem_3_0_3_decompose_comp}$,
$\eqref{lem_3_0_4_bound}$ and the fact that $\lambda(y) > \lambda_1$ we have that 

\begin{equation*}
\begin{aligned}
\int e^{-pz_T+c_1p||y_T||_1+c_2p}q_{T-s_n}(x_{s_n}, x_T)dx_T &\leq K_{n+1}e^{-pz_{s_n}}\prod_{i=1}^d \cosh\left(c_1py_{s_n}^{(i)}\right),
\end{aligned}
\end{equation*}
where $K_{n+1} = M_pC_T'\exp(c_2p+d\ln(2)+(c_1^2p^2a_2d-p\lambda_1)\xi_{n+1})$. Therefore, for $n > 0$ we have that 

\begin{equation}\label{lem_3_1_0_begin}
\begin{aligned}
&\hspace{5mm}\mathbb{E}\left[\left|e^{-Z^{\pi}_T}f(Y_T^{\pi})\prod_{i=1}^{n}\theta_{s_i - s_{i-1}}(X_{s_{i-1}}^{\pi}, X_{s_i}^{\pi})\right|^p\right]\\
&\leq K_{n+1}\int \cdots \int e^{-pz_{s_n}}\prod_{i=1}^d \cosh\left(c_1py_{s_n}^{(i)}\right)\\
&\hspace{20mm}\cdot\left(\prod_{i=1}^{n}|\theta_{s_i-s_{i-1}}(x_{s_{i-1}}, x_{s_i})|^pq_{s_i-s_{i-1}}(x_{s_{i-1}}, x_{s_i})\right)dx_{s_n}\cdots dx_{s_1}.
\end{aligned}
\end{equation}
Then we begin integrate the equation above. Starting with $x_{s_n}$ first, by using Lemma \ref{lemma_bound_thetaq} and denote $\varphi_{2\sigma_A^2\xi_{n}}(z_{s_n}-z_{s_{n-1}}-\lambda(y_{s_{n}})\xi_{n}) = q_{2\sigma_A^2\xi_{n}}(z_{s_{n-1}}, z_{s_n},y_{s_{n-1}})$, we have

\begin{equation}\label{lem_3_1_1_z}
\begin{aligned}
& \hspace{5mm} \int e^{-pz_{s_n}}\prod_{i=1}^d \cosh\left(c_1py_{s_n}^{(i)}\right)|\theta_{s_n-s_{n-1}}(x_{s_{n-1}}, x_{s_n})|^pq_{t_{n}-s_{n-1}}(x_{s_{n-1}}, x_{s_n})dx_{s_n}\\
&\leq \int e^{-pz_{s_n}}\prod_{i=1}^d \cosh\left(c_1py_{s_n}^{(i)}\right)\cfrac{C_T}{(\xi_{n})^{p/2}}\varphi_{4a_2\xi_{n}}(y_{s_n}-y_{s_{n-1}})q_{2\sigma_A^2\xi_{n}}(z_{s_{n-1}}, z_{s_n},y_{s_{n-1}})dx_{s_n}\\
&= \cfrac{C_T}{(\xi_{n})^{p/2}}\int\int e^{-pz_{s_n}}q_{2\sigma_A^2\xi_{n}}(z_{s_{n-1}}, z_{s_n},y_{s_{n-1}})dz_{s_n}\prod_{i=1}^d \cosh\left(c_1py_{s_n}^{(i)}\right)\varphi_{4a_2\xi_{n}}(y_{s_n}-y_{s_{n-1}})dy_{s_n}\\
&\leq \cfrac{C_T}{(\xi_{n})^{p/2}}M_pe^{-pz_{s_{n-1}}-p\lambda(y_{s_{n-1}})\xi_{n}}\int \prod_{i=1}^d \cosh\left(c_1py_{s_n}^{(i)}\right)\varphi_{4a_2\xi_{n}}(y_{s_n}-y_{s_{n-1}})dy_{s_n},
\end{aligned}
\end{equation}
where the last equation holds because we are using $\eqref{notation_begin}$. For the $y$ component, from Lemma \ref{absolute inequality} and Lemma \ref{integration} we know that 

\begin{equation}\label{lem_3_1_2_y}
\begin{aligned}
&\hspace{5mm}\int \prod_{i=1}^d \cosh\left(c_1py_{s_n}^{(i)}\right)\varphi_{4a_2\xi_{n}}(y_{s_n}-y_{s_{n-1}})dy_{s_n}\\
&\leq \prod_{i=1}^d \left(\int \cosh\left(c_1py_{s_n}^{(i)}\right)\varphi_{4a_2\xi_{n}}(y_{s_n}^{(i)} - y_{s_{n-1}}^{(i)})dy_{s_n}^{(i)}\right)\\
& = e^{2c_1^2p^2a_2d\xi_{n}}\prod_{i=1}^d  \cosh\left( c_1py_{s_{n-1}}^{(i)}\right).\\
\end{aligned}
\end{equation}
By combining $\eqref{lem_3_1_0_begin}$, $\eqref{lem_3_1_1_z}$ and $\eqref{lem_3_1_2_y}$, if we denote 

\begin{equation*}
\begin{aligned}
K_n = M_pC_T\exp\left((2c_1^2p^2a_2d -p\lambda_1)\xi_{n}\right),
\end{aligned}
\end{equation*}
we have 
\begin{equation}\label{lem_3_2_0_begin}
\begin{aligned}
&\hspace{5mm}\mathbb{E}\left[\left|\frac{e^{-Z^{\pi}_T+c_1||Y_T^\pi||_1+c_2}}{p_n(s_1, \cdots , s_n)}\prod_{i=1}^{n}\theta_{s_i - s_{i-1}}(X_{s_{i-1}}^{\pi}, X_{s_i}^{\pi})\right|^p\right]\\
&\leq K_{n+1}\cfrac{K_n}{\xi_{n}^{p/2}p_n(s_1, \cdots , s_n)^p}\int \cdots \int e^{-pz_{s_{n-1}}}\left(\prod_{i=1}^d\cosh(c_1py_{s_{n-1}}^{(i)})\right)\\
&\hspace{30mm}\cdot\left(\prod_{i=1}^{n-1}|\theta_{s_i-s_{i-1}}(x_{s_{i-1}}, x_{s_i})|^pq_{s_i-s_{i-1}}(x_{s_{i-1}}, x_{s_i})\right)dx_{s_{n-1}}\cdots dx_{s_1}\\
\end{aligned}
\end{equation}
We then start to integrate $x_{s_{n-1}}$. One can notice that for any $ 0 < j < n$, by following the same methodology in $\eqref{lem_3_1_1_z}$ and $\eqref{lem_3_1_2_y}$, we have
\begin{equation}\label{lem_3_2_1_begin}
\begin{aligned}
&\hspace{5mm}\int e^{-pz_{s_j}}\left(\prod_{i=1}^d\cosh(c_1py_{s_{j}}^{(i)})\right)|\theta_{s_j-s_{j-1}}(x_{s_{j-1}}, x_{s_j})|^pq_{s_j-s_{j-1}}(x_{s_{j-1}}, x_{s_j})dx_{s_j}\\
&\leq \cfrac{M_pC_T}{\xi_j^{p/2}}e^{-pz_{s_{j-1}}-p\lambda(y_{s_{j-1}})\xi_j}\int \left(\prod_{i=1}^d \cosh(c_1py_{s_j}^{(i)})\varphi_{4a_2\xi_j}(y_{s_j}^{(i)} - y_{s_{j-1}}^{(i)})\right)dy _{s_j}\\
&= \cfrac{M_pC_T}{\xi_j^{p/2}}e^{-pz_{s_{j-1}}-p\lambda(y_{s_{j-1}})\xi_j}e^{2c_1^2p^2a_2d\xi_j}\prod_{i=1}^d \cosh(c_1py_{s_{j-1}}^{(i)})\\
&\leq \cfrac{K_j}{\xi_j^{p/2}}e^{-pz_{s_{j-1}}}\prod_{i=1}^d \cosh(c_1py_{s_{j-1}}^{(i)}),\\
\end{aligned}
\end{equation}
where 
\begin{equation*}
\begin{aligned}
K_j = M_pC_T\exp((2c_1^2p^2a_2d - p\lambda_1)\xi_j).
\end{aligned}
\end{equation*}
Therefore, by adapting $\eqref{lem_3_2_1_begin}$ to $\eqref{lem_3_2_0_begin}$, we have

\begin{equation*}
\begin{aligned}
&\hspace{5mm}\mathbb{E}\left[\left|\frac{e^{-Z^{\pi}_T+c_1||Y_T^\pi||_1+c_2}}{p_n(s_1, \cdots , s_n)}\prod_{i=1}^{n}\theta_{s_i - s_{i-1}}(X_{s_{i-1}}^{\pi}, X_{s_i}^{\pi})\right|^p\right]\\
&\leq \cfrac{\prod_{j=1}^{n+1}K_j}{p_n(s_1, \cdots , s_n)^p\prod_{j=1}^{n}(s_j - s_{j-1})^{p/2}}\prod_{i=1}^d \cosh(c_1py_{0}^{(i)})\\
&\leq \cfrac{\bar{M}_T}{p_n(s_1, \cdots , s_n)^p\prod_{j=1}^{n}(s_j - s_{j-1})^{p/2}}\prod_{i=1}^d \cosh(c_1py_{0}^{(i)})
\end{aligned}
\end{equation*}
where the last line come form the fact that $$\bar{M}(T,p) :=  C_T\exp((2c_1^2p^2a_2d-p\lambda_1 + \sigma_A^2p^2)T) = \prod_{j=1}^{n+1}K_j.$$
Having showing the above, using the same argument in equation (14) of \citet{anderssonhiga2017} (notice that our case corresponds to $\zeta = 1/2$ as the ``forward case'' mentioned in the paper), we have
\begin{equation*}
    \begin{aligned}
    &\hspace{5mm}\mathbb{E}\left[\left|\frac{e^{-Z^{\pi}_T+c_1||Y_T^\pi||_1+c_2}}{p_{N_T}(\tau_1, \cdots , \tau_{N_T})}\prod_{i=1}^{N_T}\theta_{\tau_i - \tau_{i-1}}(X_{\tau_{i-1}}^{\pi}, X_{\tau_i}^{\pi})\right|^p\right]\\
    &= \sum_{n=0}^{+\infty}\int_{S^n}\mathbb{E}\left[\left|\frac{e^{-Z^{\pi}_T+c_1||Y_T^\pi||_1+c_2}}{p_n(s_1, \cdots , s_n)}\prod_{i=1}^{n}\theta_{s_i - s_{i-1}}(X_{s_{i-1}}^{\pi}, X_{s_i}^{\pi})\right|^p\right]p_n(s_1,\cdots,s_n)ds\\
    &\leq \sum_{n=0}^{+\infty}\left(\left(\bar{M}(T,p)\prod_{i=1}^d \cosh(c_1py_{0}^{(i)})\right)\int_{S^n}\frac{1}{p_n(s_1, \cdots , s_n)^{p-1}}\prod_{j=1}^{n}(s_j - s_{j-1})^{-p/2}ds\right)\\
    \end{aligned}
\end{equation*}
Since we know $\sum_{n=0}^{\infty}\int_{S^n}\frac{1}{p_n(s_1, \cdots , s_n)^{p-1}}\prod_{j=1}^{n}(s_j - s_{j-1})^{-p/2}ds$ is finite \citep[Proposition 7.3]{anderssonhiga2017}, by defining 
\begin{equation}\label{eqn_MT}
    \begin{aligned}
        M(T,p) := \bar{M}(T,p)\sum_{n=0}^{\infty}\int_{S^n}\frac{1}{p_n(s_1, \cdots , s_n)^{p-1}}\prod_{j=1}^{n}(s_j - s_{j-1})^{-p/2}ds,
    \end{aligned}
\end{equation}
we know that 
\begin{equation*}
    \begin{aligned}
    \mathbb{E}\left[\left|\frac{e^{-Z^{\pi}_T+c_1||Y_T^\pi||_1+c_2}}{p_{N_T}(\tau_1, \cdots , \tau_{N_T})}\prod_{i=1}^{N_T}\theta_{\tau_i - \tau_{i-1}}(X_{\tau_{i-1}}^{\pi}, X_{\tau_i}^{\pi})\right|^p\right] \leq M(T,p)\prod_{i=1}^d \cosh(c_1py_{0}^{(i)}),
    \end{aligned}
\end{equation*}
thereby finishing the proof.

\subsection{Proof of Lemma \ref{lem_para_expectation}}
By defining $f_K(x) = (f(x) \wedge K) \vee (-K)$ for $K > 0$, we know that 
$f_K$ will be bounded and measureable. By an argument identical to Lemma $3.1$ of \citet{chen2019} (see also 
Remark $3.2$ in that reference) we know 
\begin{equation}\label{pf_thm1_eqn_1}
\begin{aligned}
\mathbb{E}\left[e^{(-\bar{A}_T) \wedge K}f_K(Y_T)\right] &= \mathbb{E}\left[e^{(-\bar{A}_T^\pi)\wedge K}f_K(Y_T^{\pi})\Theta_2(Z^{\pi}, T)\right].
\end{aligned}
\end{equation}
Our goal is to apply dominated convergence theorem to the both side of the equation. We first focus on the limit of RHS. From Lemma \ref{lemma_para_var} we know
\begin{equation*}
\begin{aligned}
\mathbb{E}\left[\left|e^{-\bar{A}^{\pi}_T}f(Y_T^{\pi})\Theta_2(Z^{\pi}, T)\right|\right] < +\infty.\\
\end{aligned}
\end{equation*}
Henceforth by applying Dominated convergence theorem on the RHS of $\eqref{pf_thm1_eqn_1}$ we have 

\begin{equation*}
\begin{aligned}
\lim_{K\to +\infty} \mathbb{E}\left[e^{(-\bar{A}_T^\pi)\wedge K}f_K(Y_T^{\pi})\Theta_2(Z^{\pi}, T)\right]=\mathbb{E}\left[e^{-\bar{A}_T^\pi}f(Y_T^{\pi})\Theta_2(Z^{\pi}, T)\right].
\end{aligned}
\end{equation*}
For the LHS, we have 
\begin{equation*}
\begin{aligned}
e^{(-\bar{A}_T)\wedge K}f_K(Y_T) \leq e^{-\bar{A}_T}e^{c_1||Y_T||_1 + c_2}. 
\end{aligned}
\end{equation*}
Then from \citet{menozzi2021}, since $\bar{A}_T$ has at most linear growth of drift and the diffusion matrix satisfy the conditions, $Z = (Y, \bar{A})$ will have a density of gaussian type upper bound, so we have
\begin{equation*}
\begin{aligned}
\mathbb{E}\left[e^{-\bar{A}_T}e^{c_1||Y_T||_1 + c_2}\right] < +\infty. 
\end{aligned}
\end{equation*}
By applying dominated convergence theorem on both sides of $\eqref{pf_thm1_eqn_1}$ we have that 
\begin{equation*}
\begin{aligned}
\mathbb{E}\left[e^{-\bar{A}_T}f(Y_T)\right] &= \mathbb{E}\left[e^{-\bar{A}_T^\pi}f(Y_T^{\pi})\Theta_2(Z^{\pi}, T)\right],
\end{aligned}
\end{equation*}
thereby finishing the proof.

\end{document}